\documentclass[reqno]{amsart}

\usepackage[usenames,dvipsnames]{pstricks}
\usepackage{epsfig}

\usepackage{color}
\date{\today}

\usepackage[latin1]{inputenc}
\usepackage[T1]{fontenc}
\usepackage{amsmath, amsthm}
\usepackage[english]{babel}
\usepackage{amssymb}
\usepackage{latexsym}
\usepackage{graphicx}
\usepackage[all]{xy}

\newtheorem{theorem}{Theorem}[section]
\newtheorem{lemma}[theorem]{Lemma}
\newtheorem{proposition}[theorem]{Proposition}
\theoremstyle{definition}
\newtheorem{definition}[theorem]{Definition}
\newtheorem{example}[theorem]{Example}
\newtheorem{corollary}[theorem]{Corollary}

\newtheorem{remark}[theorem]{Remark}

\newcommand{\ot}{\otimes}
\newcommand{\co}{\circ}

\begin{document}

\begin{center}

{\huge{\bf Cleft and Galois extensions associated to a weak Hopf quasigroup}}

\end{center}

\ \\
\begin{center}
{\bf J.N. Alonso \'Alvarez$^{1}$, J.M. Fern\'andez Vilaboa$^{2}$, R.
Gonz\'{a}lez Rodr\'{\i}guez$^{3}$}
\end{center}

\ \\
\hspace{-0,5cm}$^{1}$ Departamento de Matem\'{a}ticas, Universidad
de Vigo, Campus Universitario Lagoas-Marcosende, E-36280 Vigo, Spain
(e-mail: jnalonso@ uvigo.es)
\ \\
\hspace{-0,5cm}$^{2}$ Departamento de \'Alxebra, Universidad de
Santiago de Compostela.  E-15771 Santiago de Compostela, Spain
(e-mail: josemanuel.fernandez@usc.es)
\ \\
\hspace{-0,5cm}$^{3}$ Departamento de Matem\'{a}tica Aplicada II,
Universidad de Vigo, Campus Universitario Lagoas-Marcosende, E-36310
Vigo, Spain (e-mail: rgon@dma.uvigo.es)
\ \\

{\bf Abstract} In this paper we introduce the notions of cleft and Galois (with normal basis) extension associated to a weak Hopf quasigroup. We show that, under suitable conditions, both notions are equivalent. As a particular instance we recover the classical results for (weak) Hopf algebras. Moreover, taking into account that weak Hopf quasigroups generalize the notion of Hopf quasigroup, we obtain the definitions of cleft and Galois (with normal basis) extension associated to a Hopf quasigroup and we get the equivalence betwen these extensions in this setting. 

\vspace{0.5cm}



\section{introduction}
The notion of Galois extension asssociated to a Hopf algebra $H$ was introduced in 1981 by Kreimer and Takeuchi  in the following way:  let $A$ be a right $H$-comodule algebra with
coaction $\rho_{A}(a)=a_{(0)}\otimes a_{(1)}$, then the extension
$A^{coH}\hookrightarrow  A$, where $A^{coH}=\{a\in A\;;\; \rho_{A}(a)=a\otimes
1_{H}\}$  is the subalgebra of coinvariant elements, is $H$-Galois if the
canonical morphism $\gamma_{A}:A\otimes_{A^{coH}}A\rightarrow A\otimes
H$, defined by $\gamma_{A}(a\otimes b)=ab_{(0)}\otimes b_{(1)}$,
is an isomorphism. This definition has its
origin in the approach to Galois theory of groups acting on
commutative rings developed by Chase, Harrison and Rosenberg and in the extension of this theory  to
coactions of a Hopf algebra $H$ acting on a commutative $k$-algebra $A$  over a commutative ring $k$, developed in
1969 by  Chase and Sweedler \cite{CSW}. An interesting class of $H$-Galois extensions  has been provided by those for which there exists a convolution invertible right $H$-comodule morphism $h:H\rightarrow A$ called the cleaving morphism. These extensions were called cleft  and it is well known that, using the notion of  normal basis introduced by  Kreimer
and Takeuchi  in \cite{KT}, Doi and Takeuchi proved in \cite{doi3} that  $A^{coH}\hookrightarrow A$ is a cleft extension if and only if it is $H$-Galois with normal basis, i.e., the extension $A^{coH}\hookrightarrow  A$ is $H$-Galois and $A$ is  isomorphic to the tensor product of   $A^{coH}$ with $H$ as left $A^{coH}$-modules and right $H$-comodules.

The result obtained by Doi and Takeuchi was generalized in \cite{ManelEmilio}  to $H$-Galois extensions for Hopf algebras living in a symmetric monoidal closed category ${\mathcal C}$ and in \cite{BRZ} Brzezi\'{n}ski proved that if $A$ is an algebra, $C$ is a coalgebra and $(A,C,\psi)$ is an entwining structure such that $A$ is an entwined module, the existence of a convolution invertible $C$-comodule morphism $h:C\rightarrow A$ is equivalent to  that $A$ is a Galois extension by the coalgebra $C$ (see \cite{BRZH} for the definition) and $A$ is isomorphic, as left $A^{coH}$-modules and right $C$-comodules, to the tensor product of the coinvariant subalgebra  $A^{coC}$ with $C$.

A more general result was proved in \cite{AFG2}, in a monoidal setting,  for weak Galois extensions associated to the weak entwining structures introduced by Caenepeel and De Groot in \cite{caengroot}. In \cite{AFG2} the  notion of weak cleft extension was defined, and Theorem 2.11  of \cite{AFG2} stated that for a weak entwining structure $(A,C,\psi)$ such that $A$ is an entwined module, if the functor $A\otimes -$ preserves coequalizers, $A$ is a weak $C$-cleft extension of the coinvariants subalgebra if and only if it is a weak $C$-Galois extension and the normal basis property, defined in \cite{AFG2}, holds. Since Galois extensions associated to weak Hopf algebras (see \cite{bohm}) are examples of weak Galois extensions, the characterization of weak cleft extensions in terms of weak Galois extensions satisfying the normal basis condition can be applied to them. Morever, this kind of result can be obtained for cleft extensions associated to lax entwining structures \cite{AFGS-1}, and for cleft extensions associated to co-extended weak entwining structures \cite{AFG-coexten}. 

The results cited in the previous paragraphs were proved in an associative setting because all the extensions are linked to Hopf algebras, to weak Hopf algebras,  or to algebraic structures related with them, i.e. entwining structures and weak entwining structures. The main motivation of this paper is to show that it is possible to obtain similar results  working in a non-associative context, that is, when we study extensions related with non-associative algebra structures like Hopf quasigroups or, more generally, like weak Hopf quasigroups.  Hopf quasigroups are a generalization of  Hopf algebras in the context of non-associative algebra, where  the lack of the associativity is compensated by some
axioms involving the antipode. The notion of Hopf quasigroup was introduced
 by Klim and Majid in \cite{Majidesfera}, in order to
understand the structure and relevant properties of the algebraic
$7$-sphere,  and is
a particular instance of  unital coassociative
$H$-bialgebra in the sense of P\'erez Izquierdo \cite{PI2}.
It includes as example
the enveloping algebra  of a Malcev algebra (see
\cite{Majidesfera} and \cite{PIS}) when the base ring has characteristic not equal to $2$ nor $3$. In this sense Hopf quasigroups extend the notion of Hopf algebra in a parallel way that Malcev algebras extend the one of Lie algebra. On the other hand, it also contains as an example the notion of quasigroup algebra 
of an I.P. loop. Therefore, Hopf quasigroups unify I.P. loops and Malcev
algebras in the same way that Hopf algebras unify groups and Lie
algebras.  On the other hand, weak Hopf quasigroups are a new Hopf algebra generalization (see \cite{AFG-Weak-quasi}) that  encompass weak Hopf algebras and Hopf quasigroups. As was proved in \cite{AFG-Weak-quasi}, the main family of non-trivial examples of these algebraic structures can be obtained  working with bigroupoids, i.e., bicategories where every $1$-cell is an equivalence and every $2$-cell is an isomorphism. 

The first result linking Hopf Galois extensions with normal basis and cleft extensions in the Hopf quasigroup setting can be found in 
\cite{AFG-3}. More specifically, in \cite{AFGS-2} we introduce the notion of cleft extension (cleft right $H$-comodule algebra) for a Hopf quasigroup $H$ in a strict monoidal category ${\mathcal C}$ with tensor product $\ot$ and unit object $K$. The notion of Galois extension with normal basis for $H$ was introduced in \cite{AFG-3}, and  we proved that, when the object of coinvariants is the unit object of the category, cleft extensions and  Galois extension with normal basis and with the inverse of the canonical morphism almost lineal, are the same.  Therefore, in \cite{AFG-3}, we extend the result proved by  Doi and Takeuchi in \cite{doi3} to the Hopf quasigroup setting, characterizing  Galois extensions with normal basis in terms of cleft extensions  when the object of coinvariants is $K$. The aim of this new paper is to show that all these results, that is, the one obtained for Hopf algebras in \cite{doi3}, the one obtained for weak Hopf algebras in \cite{AFG2}, and the  one proved for Hopf quasigroups in \cite{AFG-3}, are particular instances of a more general result that we can prove for weak Hopf quasigroups.

An outline of the paper is as follows. In Section 1 we set the general framework and review the basic properties of weak Hopf quasigroups,  in a strict symmetric monoidal category with equalizers and coequalizers, focusing in the following fact: if $H$ is a weak Hopf quasigroup and $\Pi_{H}^{L}$ is the target morphism (this morphism is defined as in the weak Hopf algebra setting), the image of $\Pi_{H}^{L}$, denoted by $H_{L}$, is a monoid, that is the restriction of the product of $H$ to $H_{L}$ is associative.  In Section 2, we introduce the notions of right $H$-comodule magma, weak $H$-Galois extension,  and weak $H$-Galois extension with normal basis, proving some technical results that we need in the following sections.  Section 3 is devoted to the study of weak $H$-cleft extensions for weak Hopf quasigroups. In particular we show that these kind of extensions contain as examples the notion of weak $H$-cleft extension associated to a weak Hopf algebra \cite{nmra1}, as well as the notion of cleft right $H$-comodule algebra introduced in \cite{AFGS-2} for Hopf quasigroups. In the last section, we can find the main result of this paper, which assures that for any right $H$-comodule magma $(A,\rho_{A})$ such that $A\ot -$ preserves coequalizers, under suitable conditions (see Theorem \ref{caracterizacion}), the following assertions are equivalent:
\begin{itemize}
\item  $A^{co H}\hookrightarrow A$ is a weak $H$-Galois extension with normal basis and the morphism $\gamma_{A}^{-1}$ is almost lineal.
\item  $A^{co H}\hookrightarrow A$ is a weak $H$-cleft extension.
\end{itemize}
In the associative setting the conditions assumed in Theorem \ref{caracterizacion} hold trivially and then it generalizes the one proved by Doi and Takeuchi for Hopf algebras in \cite{doi3}. Also, for a weak Hopf algebra $H$, we obtain an equivalence that is a particular instance of the one obtained in \cite{AFG2} for Galois extensions associated to weak entwining structures. Finally, as a corollary of Theorem \ref{caracterizacion}, we have a result for  Hopf quasigroups, which shows the close connection between the notion of cleft right $H$-comodule algebra and the one of $H$-Galois extension with normal basis introduced in this paper, improving the equivalence obtained in \cite{AFG-3} because we remove the condition $A^{co H}=K$.

\section{Weak Hopf quasigroups}

Throughout this paper $\mathcal C$ denotes a strict  symmetric monoidal category with tensor product $\ot$, unit object $K$ and natural isomorphism of symmetry $c$. For each object $M$ in  $ {\mathcal C}$, we denote the identity morphism by $id_{M}:M\rightarrow M$ and, for simplicity of notation, given objects $M$, $N$ and $P$ in ${\mathcal C}$ and a morphism $f:M\rightarrow N$, we write $P\ot f$ for $id_{P}\ot f$ and $f \ot P$ for $f\ot id_{P}$. We want to point out that there is no loss of generality in assuming that ${\mathcal C}$ is strict because by Theorem 3.5  of \cite{Christian} (which implies the Mac Lane's coherence theorem)  every monoidal category is monoidally equivalent to a strict one. This lets us to treat monoidal categories
as if they were strict and, as a consequence, the results proved in this paper hold for every non-strict symmetric monoidal category.

From now on we also assume that ${\mathcal C}$  admits equalizers and coequalizers. Then every idempotent morphism splits, i.e., for every morphism $\nabla_{Y}:Y\rightarrow Y$ such that $\nabla_{Y}=\nabla_{Y}\co\nabla_{Y}$, there exist an object $Z$ and morphisms $i_{Y}:Z\rightarrow Y$ and $p_{Y}:Y\rightarrow Z$ such that $\nabla_{Y}=i_{Y}\co p_{Y}$ and $p_{Y}\co i_{Y} =id_{Z}$.

\begin{definition}

{\rm By a unital  magma in ${\mathcal C}$ we understand a triple $A=(A, \eta_{A}, \mu_{A})$ where $A$ is an object in ${\mathcal C}$ and $\eta_{A}:K\rightarrow A$ (unit), $\mu_{A}:A\ot A \rightarrow A$ (product) are morphisms in ${\mathcal C}$ such that $\mu_{A}\co (A\ot \eta_{A})=id_{A}=\mu_{A}\co (\eta_{A}\ot A)$. If $\mu_{A}$ is associative, that is, $\mu_{A}\co (A\ot \mu_{A})=\mu_{A}\co (\mu_{A}\ot A)$, the unital magma will be called a monoid in ${\mathcal C}$.   Given two unital magmas
(monoids) $A= (A, \eta_{A}, \mu_{A})$ and $B=(B, \eta_{B}, \mu_{B})$, $f:A\rightarrow B$ is a morphism of unital magmas (monoids)  if $\mu_{B}\co (f\ot f)=f\co \mu_{A}$ and $ f\co \eta_{A}= \eta_{B}$. 

By duality, a counital comagma in ${\mathcal C}$ is a triple ${D} = (D, \varepsilon_{D}, \delta_{D})$ where $D$ is an object in ${\mathcal C}$ and $\varepsilon_{D}: D\rightarrow K$ (counit), $\delta_{D}:D\rightarrow D\ot D$ (coproduct) are morphisms in ${\mathcal C}$ such that $(\varepsilon_{D}\ot D)\co \delta_{D}= id_{D}=(D\ot \varepsilon_{D})\co \delta_{D}$. If $\delta_{D}$ is coassociative, that is, $(\delta_{D}\ot D)\co \delta_{D}= (D\ot \delta_{D})\co \delta_{D}$, the counital comagma will be called a comonoid. If ${D} = (D, \varepsilon_{D}, \delta_{D})$ and ${ E} = (E, \varepsilon_{E}, \delta_{E})$ are counital comagmas
(comonoids), $f:D\rightarrow E$ is  a morphism of counital comagmas (comonoids) if $(f\ot f)\co \delta_{D} =\delta_{E}\co f$ and  $\varepsilon_{E}\co f =\varepsilon_{D}.$

If  $A$, $B$ are unital magmas (monoids) in ${\mathcal C}$, the object $A\ot B$ is a unital  magma (monoid) in ${\mathcal C}$ where $\eta_{A\ot B}=\eta_{A}\ot \eta_{B}$ and $\mu_{A\ot B}=(\mu_{A}\ot \mu_{B})\co (A\ot c_{B,A}\ot B).$  In a dual way, if $D$, $E$ are counital comagmas (comonoids) in ${\mathcal C}$, $D\ot E$ is a  counital comagma (comonoid) in ${\mathcal C}$ where $\varepsilon_{D\ot E}=\varepsilon_{D}\ot \varepsilon_{E}$ and $\delta_{D\ot
E}=(D\ot c_{D,E}\ot E)\co( \delta_{D}\ot \delta_{E}).$

Finally, if $D$ is a comagma and $A$ a magma, given two morphisms $f,g:D\rightarrow A$ we will denote by $f\ast g$ its convolution product in ${\mathcal C}$, that is 
$$f\ast g=\mu_{A}\co (f\ot g)\co \delta_{D}.$$
}
\end{definition}

The notion of weak Hopf quasigroup in a braided monoidal category was introduced in \cite{AFG-Weak-quasi}. Now we recall this definition in our symmetric setting.

\begin{definition}
\label{Weak-Hopf-quasigroup}
{\rm A weak Hopf quasigroup $H$   in ${\mathcal C}$ is a unital magma $(H, \eta_H, \mu_H)$ and a comonoid $(H,\varepsilon_H, \delta_H)$ such that the following axioms hold:

\begin{itemize}

\item[(a1)] $\delta_{H}\co \mu_{H}=(\mu_{H}\ot \mu_{H})\co \delta_{H\ot H}.$

\item[(a2)] $\varepsilon_{H}\co \mu_{H}\co (\mu_{H}\ot H)=\varepsilon_{H}\co \mu_{H}\co (H\ot \mu_{H})$

\item[ ]$= ((\varepsilon_{H}\co \mu_{H})\ot (\varepsilon_{H}\co \mu_{H}))\co (H\ot \delta_{H}\ot H)$ 

\item[ ]$=((\varepsilon_{H}\co \mu_{H})\ot (\varepsilon_{H}\co \mu_{H}))\co (H\ot (c_{H,H}\co\delta_{H})\ot H).$

\item[(a3)]$(\delta_{H}\ot H)\co \delta_{H}\co \eta_{H}=(H\ot \mu_{H}\ot H)\co ((\delta_{H}\co \eta_{H}) \ot (\delta_{H}\co \eta_{H}))$  \item[ ]$=(H\ot (\mu_{H}\co c_{H,H})\ot H)\co ((\delta_{H}\co \eta_{H}) \ot (\delta_{H}\co \eta_{H})).$

\item[(a4)] There exists  $\lambda_{H}:H\rightarrow H$ in ${\mathcal C}$ (called the antipode of $H$) such that, if we denote the morphisms $id_{H}\ast \lambda_{H}$ by  $\Pi_{H}^{L}$ (target morphism) and $\lambda_{H}\ast id_{H}$ by $\Pi_{H}^{R}$ (source morphism),

\begin{itemize}

\item[(a4-1)] $\Pi_{H}^{L}=((\varepsilon_{H}\co \mu_{H})\ot H)\co (H\ot c_{H,H})\co ((\delta_{H}\co \eta_{H})\ot
H).$

\item[(a4-2)] $\Pi_{H}^{R}=(H\ot(\varepsilon_{H}\co \mu_{H}))\co (c_{H,H}\ot H)\co (H\ot (\delta_{H}\co \eta_{H})).$

\item[(a4-3)]$\lambda_{H}\ast \Pi_{H}^{L}=\Pi_{H}^{R}\ast \lambda_{H}= \lambda_{H}.$

\item[(a4-4)] $\mu_H\co (\lambda_H\ot \mu_H)\co (\delta_H\ot H)=\mu_{H}\co (\Pi_{H}^{R}\ot H).$

\item[(a4-5)] $\mu_H\co (H\ot \mu_H)\co (H\ot \lambda_H\ot H)\co (\delta_H\ot H)=\mu_{H}\co (\Pi_{H}^{L}\ot H).$

\item[(a4-6)] $\mu_H\co(\mu_H\ot \lambda_H)\co (H\ot \delta_H)=\mu_{H}\co (H\ot \Pi_{H}^{L}).$

\item[(a4-7)] $\mu_H\co (\mu_H\ot H)\co (H\ot \lambda_H\ot H)\co (H\ot \delta_H)=\mu_{H}\co (H\ot \Pi_{H}^{R}).$

\end{itemize}

\end{itemize}

Note that, if in the previous definition the triple $(H, \eta_H, \mu_H)$ is a monoid, we obtain the notion of weak Hopf algebra in a symmetric monoidal category. Then, if ${\mathcal C}$ is the category of vector spaces over a field ${\Bbb F}$, we have the monoidal version of the original definition of weak Hopf algebra introduced by B\"{o}hm, Nill and Szlach\'anyi in \cite{bohm}. On the other hand, under these conditions, if  $\varepsilon_H$ and $\delta_H$ are  morphisms of unital magmas (equivalently, $\eta_{H}$, $\mu_{H}$ are morphisms of counital comagmas), $\Pi_{H}^{L}=\Pi_{H}^{R}=\eta_{H}\ot \varepsilon_{H}$. As a consequence, conditions (a2), (a3), (a4-1)-(a4-3) trivialize, and we get the notion of Hopf quasigroup defined  by Klim and Majid in \cite{Majidesfera} in the category of vector spaces over a field ${\Bbb F}$.
}

\end{definition}

\begin{example} 

\label{main-example}

{\rm It is possible to obtain non-trivial examples of weak Hopf quasigroups by working with bicategories in the sense of B\'enabou \cite{BEN}. We give a brief summary of this construction. The interested reader can see the complete details in \cite{AFG-Weak-quasi}. A bicategory ${\mathcal B}$ consists of:
\begin{itemize}
\item A set ${\mathcal B}_{0}$, whose elements $x$ are called $0$-cells.
\item For each $x$, $y\in {\mathcal B}_{0}$, a category ${\mathcal B}(x,y)$ whose objects $f:x\rightarrow y$ are called $1$-cells and whose morphisms $\alpha:f \Rightarrow g$ are called $2$-cells. The composition of $2$-cells is called the vertical composition of $2$-cells and if $f$ is a $1$-cell in ${\mathcal B}(x,y)$, $x$ is called the source of $f$, represented by $s(f)$, and $y$ is called the target of $f$, denoted by $t(f)$. 
\item For each $x\in {\mathcal B}_{0}$, an object $1_{x}\in {\mathcal B}(x,x)$, called the identity of $x$; and for each $x,y,z\in {\mathcal B}_{0}$, a functor 
$${\mathcal B}(y,z)\times {\mathcal B}(x,y)\rightarrow {\mathcal B}(x,z)$$ 
which in objects is called the $1$-cell composition $(g,f)\mapsto g\co f$, and on arrows is called horizontal composition of $2$-cells: 
$$f,f^{\prime}\in {\mathcal B}(x,y), \;\; g,g^{\prime}\in {\mathcal B}(y,z), \; \alpha:f \Rightarrow f^{\prime}, \; \beta:g \Rightarrow g^{\prime}$$
$$(\beta, \alpha)\mapsto \beta\bullet \alpha:g\co f \Rightarrow g^{\prime}\co f^{\prime}$$ 
\item For each $f\in {\mathcal B}(x,y)$, $g\in {\mathcal B}(y,z)$ and  $h\in {\mathcal B}(z,w)$, an associative isomorphism $\xi_{h,g,f}: (h\co g)\co f\Rightarrow h\co (g\co f)$; and for each $1$-cell $f$,  unit  isomorphisms $l_{f}:1_{t(f)}\co f\Rightarrow f$, $r_{f}:f\co 1_{s(f)}\Rightarrow f$, satisfying the following coherence axioms:
\begin{itemize}
\item The morphism $\xi_{h,g,f}$ is natural in $h$, $f$ and $g$ and $l_{f}$, $r_{f}$ are natural in $f$.
\item Pentagon axiom: $ \xi_{k,h,g\co f}\co \xi_{k\co h,g, f}=(id_{k}\bullet \xi_{ h,g, f})\co 
\xi_{k, h\co g, f}\co (\xi_{k,h,g}\bullet id_{f}).$ 
\item Triangle axiom: $r_{g}\bullet id_{f}=(id_{g}\bullet l_{f})\co \xi_{g,1_{t(f)},f}.$ 

\end{itemize}
\end{itemize}
A bicategory is normal if the unit  isomorphisms are identities. Every bicategory is biequivalent to a normal one.  A $1$-cell $f$ is called an equivalence if there exists a $1$-cell $g:t(f)\rightarrow s(f)$ and two isomorphisms $g\co f\Rightarrow 1_{s(f)}$, $f\co g\Rightarrow 1_{t(f)}$.  In this case we will say that $g\in Inv(f)$ and, equivalently, $f\in Inv(g)$.

A bigroupoid is a bicategory where every $1$-cell is an equivalence and every $2$-cell is an isomorphism. We will say that a bigroupoid ${\mathcal B}$ is finite if ${\mathcal B}_{0}$ is finite and ${\mathcal B}(x,y)$ is small for all $x,y$. Note that if ${\mathcal B}$ is a  bigroupoid where ${\mathcal B}(x,y)$ is small for all $x,y$, and we pick up a finite number of $0$-cells, considering the full sub-bicategory  generated by these $0$-cells, we have an example of finite bigroupoid.

Let ${\mathcal B}$ be a finite normal bigroupoid and denote  by ${\mathcal B}_{1}$ the set of $1$-cells. Let ${\Bbb F}$ be a field and ${\Bbb F}{\mathcal B}$ the direct product 
$${\Bbb F}{\mathcal B}=\bigoplus_{f\in {\mathcal B}_{1}}Ff.$$
The vector space ${\Bbb F}{\mathcal B}$ is a unital nonassociative algebra  where the product of two $1$-cells is equal to their $1$-cell composition if the latter is defined and $0$ otherwise, i.e.,
$g.f=g\circ f$ if $s(g)=t(f)$ and
$g.f=0$ if $s(g)\neq t(f)$. The unit element is $$1_{{\Bbb F}{\mathcal B}}=\sum_{x\in {\mathcal B}_{0}}1_{x}.$$

Let $H={\Bbb F}{\mathcal B}/I({\mathcal B})$ be the quotient algebra where  $I({\mathcal B})$ is the ideal of ${\Bbb F}{\mathcal B}$ generated  by 
$$ h-g\co (f\co h),\; p-(p\co f)\co g,$$ 
with  $f\in {\mathcal B}_{1},$ 
$g\in Inv(f)$,  and $h,p \in {\mathcal B}_{1}$ such that $t(h)=s(f)$, $t(f)=s(p)$. In what follows, for any $1$-cell $f$ we denote its class in  $H$  by $[f]$. If we assume that $I({\mathcal B})$ is a proper ideal and for $[f]$ we define $[f]^{-1}$ by the class of $g\in Inv(f)$, we obtain  that $[f]^{-1}$ is well-defined. Therefore the vector space $H$ with the product $\mu_{H}([g]\ot [f])=[g.f]$ and  the  unit $$\eta_{ H}(1_{{\Bbb F}})=[1_{{\Bbb F}{\mathcal B}}]=\sum_{x\in {\mathcal B}_{0}}[1_{x}]$$ is a unital magma. Also, it is easy to show that $H$ is a comonoid with coproduct $\delta_{H}([f])=[f]\ot [f]$ and counit $\varepsilon_{H}([f])=1_{{\Bbb F}}$. Moreover,  the antipode is defined by  $\lambda_{H}:H\rightarrow H$, $\lambda_{H}([f])=[f]^{-1}$  and $H=(H,\eta_{H}, \mu_{H}, \varepsilon_{H}, \delta_{H}, \lambda_{H})$ is a weak Hopf quasigroup. Note that, in this example, if ${\mathcal B}_{0}=\{x\}$ we obtain that $H$ is a Hopf quasigroup. Moreover, if $\vert {\mathcal B}_{0}\vert >1$ and the product defined in $H$ is associative we have an example of weak Hopf algebra. 

}

\end{example}

In the end of this section we recall some properties of weak Hopf quasigroups we will need in what sequel. The interested reader can see the proofs in \cite{AFG-Weak-quasi}.

First note that, by Propositions 3.1 and 3.2 of \cite{AFG-Weak-quasi}, the following equalities 
\begin{equation}
\label{pi-l}
\Pi_{H}^{L}\ast id_{H}=id_{H}\ast \Pi_{H}^{R}=id_{H},
\end{equation}
\begin{equation}
\label{pi-eta}
\Pi_{H}^{L}\co\eta_{H}=\eta_{H}=\Pi_{H}^{R}\co\eta_{H},
\end{equation}
\begin{equation}
\label{pi-varep}
\varepsilon_{H}\co \Pi_{H}^{L}=\varepsilon_{H}=\varepsilon_{H}\co \Pi_{H}^{R}.
\end{equation}
hold, the antipode of a  weak Hopf quasigroup $H$ is unique and  $\lambda_{H}\co \eta_{H}=\eta_{H}$,  $\varepsilon_{H}\co\lambda_{H}=\varepsilon_{H}$. Moreover, if we define the morphisms $\overline{\Pi}_{H}^{L}$ and $\overline{\Pi}_{H}^{R}$ by 
$$\overline{\Pi}_{H}^{L}=(H\ot (\varepsilon_{H}\co \mu_{H}))\co ((\delta_{H}\co \eta_{H})\ot H)$$
and 
$$\overline{\Pi}_{H}^{R}=((\varepsilon_{H}\co \mu_{H})\ot H)\co (H\ot (\delta_{H}\co \eta_{H})),$$
in Proposition 3.4 of \cite{AFG-Weak-quasi}, we proved that $\Pi_{H}^{L}$, $\Pi_{H}^{R}$, $\overline{\Pi}_{H}^{L}$ and 
$\overline{\Pi}_{H}^{R}$ are idempotent.

On the other hand, Propositions 3.5, 3.7 and 3.9 of \cite{AFG-Weak-quasi} assert that 
\begin{equation}
\label{mu-pi-l}
\mu_{H}\co (H\ot \Pi_{H}^{L})=((\varepsilon_{H}\co
\mu_{H})\ot H)\co (H\ot c_{H,H})\co (\delta_{H}\ot
H),
\end{equation}
\begin{equation}
\label{mu-pi-r}
\mu_{H}\co (\Pi_{H}^{R}\ot H)=(H\ot(\varepsilon_{H}\co \mu_{H}))\co (c_{H,H}\ot H)\co
(H\ot \delta_{H}),
\end{equation}
\begin{equation}
\label{mu-pi-l-var}
\mu_{H}\co (H\ot \overline{\Pi}_{H}^{L})=(H\ot (\varepsilon_{H}\co
\mu_{H}))\co (\delta_{H}\ot H),
\end{equation}
\begin{equation}
\label{mu-pi-r-var}
\mu_{H}\co (\overline{\Pi}_{H}^{R}\ot H)=((\varepsilon_{H}\co
\mu_{H})\ot H)\co (H\ot \delta_{H}), 
\end{equation}
\begin{equation}
\label{delta-pi-l}
 (H\ot \Pi_{H}^{L})\co \delta_{H}=(\mu_{H}\ot H)\co (H\ot c_{H,H})\co ((\delta_{H}\co \eta_{H})\ot
H),
\end{equation}
\begin{equation}
\label{delta-pi-r}
(\Pi_{H}^{R}\ot H)\co \delta_{H}=(H\ot \mu_{H})\co (c_{H,H}\ot H)\co
(H\ot (\delta_{H}\co \eta_{H})),
\end{equation}
\begin{equation}
\label{delta-pi-l-var}
(\overline{\Pi}_{H}^{L}\ot H)\co \delta_{H}=(H\ot 
\mu_{H})\co ((\delta_{H}\co \eta_{H})\ot H),
\end{equation}
\begin{equation}
\label{delta-pi-r-var}
 (H\ot \overline{\Pi}_{H}^{R})\co \delta_{H}=(
\mu_{H}\ot H)\co (H\ot (\delta_{H}\co \eta_{H})), 
\end{equation}
\begin{equation}
\label{pi-delta-mu-pi-1}
\Pi^{L}_{H}\circ \mu_{H}\circ (H\ot
\Pi^{L}_{H})=\Pi^{L}_{H}\circ \mu_{H},
\end{equation}
\begin{equation}
\label{pi-delta-mu-pi-2}
\Pi^{R}_{H}\circ
\mu_{H}\circ (\Pi^{R}_{H}\ot H)=\Pi^{R}_{H}\circ \mu_{H},
\end{equation}
\begin{equation}
\label{pi-delta-mu-pi-3}
(H\ot \Pi^{L}_{H})\circ \delta_{H}\circ
\Pi^{L}_{H}=\delta_{H}\circ \Pi^{L}_{H},
\end{equation}
\begin{equation}
\label{pi-delta-mu-pi-4}
( \Pi^{R}_{H}\ot
H)\circ \delta_{H}\circ \Pi^{R}_{H}=\delta_{H}\circ \Pi^{R}_{H}, 
\end{equation}
hold. 

Also, it is possible to show the following identities involving the idempotent morphisms $\Pi_{H}^{L}$, $\Pi_{H}^{R}$, $\overline{\Pi}_{H}^{L}$,  $\overline{\Pi}_{H}^{R}$ and the antipode $\lambda_{H}$  (see  Propositions 3.11 and 3.12 of \cite{AFG-Weak-quasi}): 
\begin{equation}
\label{pi-composition-1}
\Pi_{H}^{L}\co
\overline{\Pi}_{H}^{L}=\Pi_{H}^{L},\;\;\;\Pi_{H}^{L}\co
\overline{\Pi}_{H}^{R}=\overline{\Pi}_{H}^{R},
\end{equation}
\begin{equation}
\label{pi-composition-2}
\overline{\Pi}_{H}^{L}\co
\Pi_{H}^{L}=\overline{\Pi}_{H}^{L},\;\;\;\overline{\Pi}_{H}^{R}\co
\Pi_{H}^{L}=\Pi_{H}^{L},
\end{equation}
\begin{equation}
\label{pi-composition-3}
\Pi_{H}^{R}\co
\overline{\Pi}_{H}^{L}=\overline{\Pi}_{H}^{L},\;\;\;
\Pi_{H}^{R}\co
\overline{\Pi}_{H}^{R}=\Pi_{H}^{R},
\end{equation}
\begin{equation}
\label{pi-composition-4}
\overline{\Pi}_{H}^{L}\co
\Pi_{H}^{R}=\Pi_{H}^{R},\;\;\; \overline{\Pi}_{H}^{R}\co
\Pi_{H}^{R}=\overline{\Pi}_{H}^{R}, 
\end{equation}
\begin{equation}
\label{pi-antipode-composition-1}
\Pi_{H}^{L}\co \lambda_{H}=\Pi_{H}^{L}\co
\Pi_{H}^{R}= \lambda_{H}\co \Pi_{H}^{R},
\end{equation}
\begin{equation}
\label{pi-antipode-composition-2}
\Pi_{H}^{R}\co
\lambda_{H}=\Pi_{H}^{R}\co \Pi_{H}^{L}= \lambda_{H}\co
\Pi_{H}^{L},
\end{equation}
\begin{equation}
\label{pi-antipode-composition-3}
\Pi_{H}^{L}=\overline{\Pi}_{H}^{R}\co
\lambda_{H}=\lambda_{H}
\co\overline{\Pi}_{H}^{L},
\end{equation}
\begin{equation}
\label{pi-antipode-composition-4}
\Pi_{H}^{R}=
\overline{\Pi}_{H}^{L}\co \lambda_{H}=\lambda_{H} \co
\overline{\Pi}_{H}^{R}. 
\end{equation}

Moreover, by Proposition 3.16  of \cite{AFG-Weak-quasi},  the equalities 
\begin{equation}
\label{mu-assoc-1}
\mu_{H}\co (\mu_{H}\ot H)\co (H\ot ((\Pi_{H}^{L}\ot H)\co \delta_{H}))=\mu_{H}=
\mu_{H}\co (\mu_{H}\ot \Pi_{H}^{R})\co (H\ot  \delta_{H}),
\end{equation}
\begin{equation}
\label{mu-assoc-2}
\mu_{H}\co (\Pi_{H}^{L}\ot \mu_{H})\co (\delta_{H}\ot H)=\mu_{H}=
\mu_{H}\co (H\ot (\mu_{H}\co ( \Pi_{H}^{R}\ot H)))\co (\delta_{H}\ot H),
\end{equation}
\begin{equation}
\label{mu-assoc-3}
\mu_{H}\co (\lambda_{H}\ot (\mu_{H}\co ( \Pi_{H}^{L}\ot H)))\co (\delta_{H}\ot H)=\mu_{H}\co (\lambda_{H}\ot H)\end{equation}
$$=
\mu_{H}\co (\Pi_{H}^{R}\ot (\mu_{H}\co ( \lambda_{H}\ot H)))\co (\delta_{H}\ot H),
$$
\begin{equation}
\label{mu-assoc-4}
\mu_{H}\co (\mu_{H}\ot H)\co (H\ot ((\lambda_{H}\ot\Pi_{H}^{L})\co \delta_{H}))=\mu_{H}\co (H\ot \lambda_{H})\end{equation}
$$= \mu_{H}\co (\mu_{H}\ot H)\co (H\ot ((\Pi_{H}^{R}\ot \lambda_{H})\co \delta_{H})),$$
hold and we  have that
\begin{equation}
\label{2-mu-delta-pi-l}
(\mu_{H}\ot (\mu_{H}\co (H\ot \Pi_{H}^{L})))\co \delta_{H\ot H}=(\mu_{H}\ot H)\co (H\ot c_{H,H})\co (\delta_{H}\ot H),
\end{equation}
\begin{equation}
\label{2-mu-delta-pi-r}
((\mu_{H}\co (\Pi_{H}^{R}\ot H))\ot \mu_{H})\co \delta_{H\ot H}=(H\ot \mu_{H})\co (c_{H,H}\ot H)\co (H\ot \delta_{H}). 
\end{equation}

Therefore (see Theorem 3.19 of \cite{AFG-Weak-quasi}), for any weak Hopf quasigroup $H$ the antipode of $H$ is antimultiplicative and anticomultiplicative, i.e.,
\begin{equation}
\label{anti-antipode-1}
\lambda_{H}\co \mu_{H}=\mu_{H}\co c_{H,H}\co (\lambda_{H}\ot
\lambda_{H}),
\end{equation}
\begin{equation}
\label{anti-antipode-2}
\delta_{H}\co \lambda_{H}=(\lambda_{H}\ot \lambda_{H})\co
c_{H,H}\co \delta_{H}. 
\end{equation}

Finally, if $H_{L}=Im(\Pi_{H}^{L})$ and 
$p_{L}:H\rightarrow H_{L}$ and $i_{L}:H_{L}\rightarrow H$ are the
morphisms such that $\Pi_{H}^{L}=i_{L}\co p_{L}$ and $p_{L}\co
i_{L}=id_{H_{L}}$, 
$$
\setlength{\unitlength}{3mm}
\begin{picture}(30,4)
\put(3,2){\vector(1,0){4}} \put(11,2.5){\vector(1,0){10}}
\put(11,1.5){\vector(1,0){10}} \put(1,2){\makebox(0,0){$H_{L}$}}
\put(9,2){\makebox(0,0){$H$}} \put(24,2){\makebox(0,0){$H\ot H$}}
\put(5.5,3){\makebox(0,0){$i_{L}$}} \put(16,3.5){\makebox(0,0){$
\delta_{H}$}} \put(16,0.5){\makebox(0,0){$(H\ot \Pi_{H}^{L}) \co
\delta_{H}$}}
\end{picture}
$$
is an equalizer diagram  and
$$
\setlength{\unitlength}{1mm}
\begin{picture}(101.00,10.00)
\put(20.00,8.00){\vector(1,0){25.00}}
\put(20.00,4.00){\vector(1,0){25.00}}
\put(55.00,6.00){\vector(1,0){21.00}}
\put(32.00,11.00){\makebox(0,0)[cc]{$\mu_{H}$ }}
\put(33.00,0.00){\makebox(0,0)[cc]{$\mu_{H}\co (H\ot \Pi_{H}^{L})
$ }} \put(65.00,9.00){\makebox(0,0)[cc]{$p_{L}$ }}
\put(13.00,6.00){\makebox(0,0)[cc]{$ H\otimes H$ }}
\put(50.00,6.00){\makebox(0,0)[cc]{$ H$ }}
\put(83.00,6.00){\makebox(0,0)[cc]{$H_{L} $ }}
\end{picture}
$$
is a coequalizer diagram. As a consequence, $(H_{L},
\eta_{H_{L}}=p_{L}\co \eta_{H}, \mu_{H_{L}}=p_{L}\co \mu_{H}\co
(i_{L}\ot i_{L}))$ is a unital magma in ${\mathcal C}$ and $(H_{L},
\varepsilon_{H_{L}}=\varepsilon_{H}\co i_{L}, \delta_{H}=(p_{L}\ot
p_{L})\co \delta_{H}\co i_{L})$ is a comonoid in ${\mathcal C}$ (see Proposition 3.13  of \cite{AFG-Weak-quasi}). 

If $H$ is the weak Hopf quasigroup defined in Example \ref{main-example} note that 
$H_{L}=\langle [1_{x}], \; x\in {\mathcal B}_{0}\rangle$. Then, in this case we have that  the induced product $\mu_{H_{L}}$ is associative because $[1_{x}].([1_{y}].[1_{z}])$ and $([1_{x}].[1_{y}]).[1_{z}]$ are equal to 
$[1_{x}]$ if $x=y=z$ and $0$ otherwise. Surprisingly, the associativity of the product $\mu_{H_{L}}$ is a general property:

\begin{proposition}
\label{monoid-hl}
Let $H$ be a weak Hopf quasigroup. The following identities hold:
\begin{equation}
\label{monoid-hl-1}
\mu_{H}\co ((\mu_{H}\co (i_{L}\ot H))\ot H)=\mu_{H}\co (i_{L}\ot \mu_{H}), 
\end{equation}
\begin{equation}
\label{monoid-hl-2}
\mu_{H}\co (H\ot (\mu_{H}\co (i_{L}\ot H)))=\mu_{H}\co ((\mu_{H}\co (H\ot i_{L}))\ot H), 
\end{equation}
\begin{equation}
\label{monoid-hl-3}
\mu_{H}\co (H\ot (\mu_{H}\co (H\ot i_{L})))=\mu_{H}\co (\mu_{H}\ot i_{L}). 
\end{equation}
As a consequence, the unital magma $H_{L}$ is a monoid in ${\mathcal C}$.
\end{proposition}

\begin{proof} First we will  prove that 
\begin{equation}
\label{aux-1-monoid-hl}
\delta_{H}\co \mu_{H}\co (i_{L}\ot H)=(\mu_{H}\ot H)\co (i_{L}\ot \delta_{H}), 
\end{equation}
\begin{equation}
\label{aux-2-monoid-hl}
\delta_{H}\co \mu_{H}\co (H\ot i_{L})=(\mu_{H}\ot H)\co (H\ot c_{H,H})\co  (\delta_{H}\ot i_{L}). 
\end{equation}
Indeed: 
\begin{itemize}
\item[ ]$\hspace{0.38cm}\delta_{H}\co \mu_{H}\co (i_{L}\ot H)$

\item[ ]$=(\mu_{H}\ot \mu_{H})\co \delta_{H\ot H}\co (i_{L}\ot H) $

\item[ ]$=(\mu_{H}\ot (\mu_{H}\co (\overline{\Pi}_{H}^{R}\ot H))\co \delta_{H\ot H}\co (i_{L}\ot H) $

\item[ ]$=(\mu_{H}\ot (((\varepsilon_{H}\co \mu_{H})\ot H))\co (H\ot \delta_{H}))\co \delta_{H\ot H}\co (i_{L}\ot H)$
\item[ ]$=(\mu_{H}\ot H)\co (i_{L}\ot \delta_{H}) .$
\end{itemize}

The first   equality follows by (a1) of Definition \ref{Weak-Hopf-quasigroup}. The second one follows by Remark 3.15 of \cite{AFG-Weak-quasi} and the third one by (\ref{mu-pi-r-var}). Finally, the fourth one is a consequence of the coassociativity of $\delta_{H}$ and  (a1) of Definition \ref{Weak-Hopf-quasigroup}.

On the other hand,  by (a1) of Definition \ref{Weak-Hopf-quasigroup}, (\ref{pi-delta-mu-pi-3}), (\ref{mu-pi-l}) and the coassociativity of $\delta_{H}$, we obtain (\ref{aux-2-monoid-hl}) because

\begin{itemize}
\item[ ]$\hspace{0.38cm} \delta_{H}\co \mu_{H}\co (H\ot i_{L})$

\item[ ]$=(\mu_{H}\ot \mu_{H})\co \delta_{H\ot H}\co (H\ot i_{L})  $

\item[ ]$=(\mu_{H}\ot (\mu_{H}\co (H\ot \Pi_{H}^{L})))\co \delta_{H\ot H}\co (H\ot i_{L}) $

\item[ ]$=(\mu_{H}\ot (((\varepsilon_{H}\co \mu_{H})\ot H)\co (H\ot c_{H,H})\co (\delta_{H}\ot H)))\co \delta_{H\ot H}\co (H\ot i_{L})  $

\item[ ]$=(\mu_{H}\ot H)\co (H\ot c_{H,H})\co  (\delta_{H}\ot i_{L}) .$
\end{itemize}

Then, (\ref{monoid-hl-1}) holds because

\begin{itemize}
\item[ ]$\hspace{0.38cm} \mu_{H}\co ((\mu_{H}\co (i_{L}\ot H))\ot H)$

\item[ ]$= \mu_{H}\co ((\mu_{H}\co (H\ot \Pi_{H}^{L}))\ot H)\co ( (\mu_{H}\co (i_{L}\ot H)) \ot \delta_{H})$

\item[ ]$= (\varepsilon_{H}\ot H)\co \mu_{H\ot H}\co ((\delta_{H}\co \mu_{H}\co (i_{L}\ot H))\ot \delta_{H})$

\item[ ]$= (\varepsilon_{H}\ot H)\co \mu_{H\ot H}\co (((\mu_{H}\ot H)\co (i_{L}\ot \delta_{H}))\ot \delta_{H}) $

\item[ ]$=(\varepsilon_{H}\ot H)\co (\mu_{H}\ot H)\co (i_{L}\ot ((\mu_{H}\ot \mu_{H})\co \delta_{H\ot H}))  $

\item[ ]$= (\varepsilon_{H}\ot H)\co (\mu_{H}\ot H)\co (i_{L}\ot (\delta_{H}\co \mu_{H})) $

\item[ ]$=(\varepsilon_{H}\ot H) \co  \delta_{H}\co \mu_{H}\co (i_{L}\ot \mu_{H}) $

\item[ ]$= \mu_{H}\co (i_{L}\ot \mu_{H}).$

\end{itemize}
 
The first equality follows by (\ref{mu-assoc-1}), the second one by (\ref{mu-pi-l}) and the third  and sixth ones  by (\ref{aux-1-monoid-hl}). The fourth one is a consequence of (a2) of Definition \ref{Weak-Hopf-quasigroup}. In the fifth one we used (a1) of Definition \ref{Weak-Hopf-quasigroup} and the last one relies on the properties of the counit.

The proof for (\ref{monoid-hl-2}) is the following: 

\begin{itemize}
\item[ ]$\hspace{0.38cm} \mu_{H}\co (H\ot (\mu_{H}\co (i_{L}\ot H))$

\item[ ]$= \mu_{H}\co ((\mu_{H}\co (H\ot \Pi_{H}^{L}))\ot H)\co (H\ot (\delta_{H}	\co \mu_{H}\co (i_{L}\ot H)))$

\item[ ]$= \mu_{H}\co ((\mu_{H}\co (H\ot \Pi_{H}^{L}))\ot H)\co (H\ot \mu_{H}\ot H)\co (H\ot i_{L}\ot \delta_{H})$

\item[ ]$=(\varepsilon_{H}\ot H)\co\mu_{H\ot H}\co (\delta_{H}\ot ((\mu_{H}\ot H)\co (i_{L}\ot \delta_{H})))  $

\item[ ]$=((\varepsilon_{H}\co \mu_{H})\ot \mu_{H})\co (H\ot  c_{H,H}\ot H)\co (((\mu_{H}\ot H)\co (H\ot c_{H,H})\co (\delta_{H}\ot i_{L}))\ot \delta_{H})$

\item[ ]$=((\varepsilon_{H}\co \mu_{H})\ot \mu_{H})\co ( H\ot c_{H,H}\ot H)\co ((\delta_{H}\co \mu_{H}\co (H\ot i_{L}))\ot \delta_{H}) $

\item[ ]$= (\varepsilon_{H}\ot H)\co\delta_{H}\co \mu_{H}\co ((\mu_{H}\co (H\ot i_{L}))\ot H)$

\item[ ]$= \mu_{H}\co ((\mu_{H}\co (H\ot i_{L}))\ot H).$

\end{itemize}

The first equality is a consequence of (\ref{mu-assoc-1}), the second one follows by (\ref{aux-1-monoid-hl}) and in the third one we used (\ref{mu-pi-l}). The fourth equality relies on the naturalness of $c$ and   (a2) of Definition \ref{Weak-Hopf-quasigroup}. The fifth one follows from (\ref{aux-2-monoid-hl}), in the sixth equality we applied (a1) of Definition \ref{Weak-Hopf-quasigroup} and the last one follows by the properties of the counit.

Similarly, we will prove (\ref{monoid-hl-3}). Indeed:

\begin{itemize}
\item[ ]$\hspace{0.38cm}\mu_{H}\co (H\ot (\mu_{H}\co (H\ot i_{L})) $

\item[ ]$=\mu_{H}\co ((\mu_{H}\co (H\ot \Pi_{H}^{L}))\ot H)\co (H\ot (\delta_{H}	\co \mu_{H}\co (H\ot i_{L}))) $

\item[ ]$=\mu_{H}\co ((\mu_{H}\co (H\ot \Pi_{H}^{L}))\ot H)\co (H\ot ((\mu_{H}\ot H)\co (H\ot c_{H,H})\co  (\delta_{H}\ot i_{L}))) $

\item[ ]$= (\varepsilon_{H}\ot H)\co\mu_{H\ot H}\co (\delta_{H}\ot ((\mu_{H}\ot H)\co (H\ot c_{H,H})\co  (\delta_{H}\ot i_{L}))) $

\item[ ]$=(\varepsilon_{H}\ot H)\co (\mu_{H}\ot \mu_{H})\co ( \mu_{H}\ot c_{H,H}\ot H)\co (H\ot c_{H,H}\ot c_{H,H})\co 
(\delta_{H}\ot \delta_{H}\ot i_{L}) $

\item[ ]$=((\varepsilon_{H}\co \mu_{H})\ot H)\co (H\ot c_{H,H})\co ((\delta_{H}\co \mu_{H})\ot i_{L}) $

\item[ ]$= \mu_{H}\co (\mu_{H}\ot (\Pi_{H}^{L}\co i_{L}))$

\item[ ]$=\mu_{H}\co (\mu_{H}\ot i_{L}) .$

\end{itemize}

The first equality follows by (\ref{mu-assoc-1}), the second one by (\ref{aux-2-monoid-hl}) and the third one by (\ref{mu-pi-l}). The fourth one is a consequence of the naturalness of $c$ and (a2) of Definition \ref{Weak-Hopf-quasigroup}. In the fifth one we used (a1) of Definition \ref{Weak-Hopf-quasigroup}, the sixth one follows by (\ref{aux-2-monoid-hl}) and the last one relies on the properties of $\Pi_{H}^{L}$.

Finally, by Proposition 3.9 of \cite{AFG-Weak-quasi}, (\ref{monoid-hl-2}) and the equality
\begin{equation}
\label{pi-mu-pi-pi}
\Pi_{H}^{L}\co \mu_{H}\co (\Pi_{H}^{L}\ot \Pi_{H}^{L})=\mu_{H}\co (\Pi_{H}^{L}\ot \Pi_{H}^{L}),
\end{equation}
 it is easy to show that $\mu_{H_{L}}\co (H_{L}\ot \mu_{H_{L}})=\mu_{H_{L}}\co (\mu_{H_{L}}\ot H_{L})$ and therefore the unital magma $H_{L}$ is a monoid in ${\mathcal C}$. Note that (\ref{pi-mu-pi-pi}) holds because, by (\ref{mu-pi-l}), (\ref{pi-delta-mu-pi-3}) and the naturalness of $c$, we have 
 $$\mu_{H}\co (\Pi_{H}^{L}\ot \Pi_{H}^{L})=((\varepsilon_{H}\co \mu_{H})\ot H)\co (H\ot c_{H,H})\co ((\delta_{H}\co \Pi_{H}^{L})\ot H)$$
 $$=((\varepsilon_{H}\co \mu_{H})\ot H)\co (H\ot c_{H,H})\co (((H\ot \Pi_{H}^{L})\co \delta_{H}\co \Pi_{H}^{L})\ot H)=\Pi_{H}^{L}\co \mu_{H}\co (\Pi_{H}^{L}\ot \Pi_{H}^{L}).$$

\end{proof}

\section{Galois extensions associated to weak Hopf quasigroups}

In this section we introduce the notion of Galois extension (with normal basis) associated to a weak Hopf quasigroup that generalizes the one defined  for Hopf algebras in \cite{KT} and for weak Hopf algebras in \cite{AFG2}. Moreover, if we consider that $\varepsilon_H$ and $\delta_H$ are  morphisms of unital magmas, $H$ is a Hopf quasigroup and we get a definition of Galois (with normal basis) extension associated to a Hopf quasigroup.

\begin{definition}
\label{H-comodulomagma}
{\rm  Let $H$ be a weak Hopf quasigroup
and let $(A, \rho_{A})$ be a unital magma (monoid), which is also a right
$H$-comodule (i.e., $(A\ot \varepsilon_{H})\co \rho_{A}=id_{A}$, $(\rho_{A}\ot H)\co \rho_{A}=(A\ot \delta_{H})\co \rho_{A}$), such that 
\begin{equation}
\label{chmagma}
\mu_{A\ot H}\co (\rho_{A}\ot
\rho_{A})=\rho_{A}\co \mu_{A}.
\end{equation}
 We will say that $A$ is a right $H$-comodule
magma (monoid) if any of the following equivalent conditions hold:

\begin{itemize}
\item[(b1)]$(\rho_{A}\ot H)\co \rho_{A}\co
\eta_{A}=(A\ot (\mu_{H}\co c_{H,H})\ot H)\co
((\rho_{A}\co \eta_{A})\ot (\delta_{H}\co \eta_{H})). $
\item[(b2)]$(\rho_{A}\ot H)\co \rho_{A}\co
\eta_{A}=(A\ot \mu_{H}\ot H)\co ((\rho_{A}\co \eta_{A})\ot (\delta_{H}\co \eta_{H})). $
\item[(b3)]$(A\ot \overline{\Pi}_{H}^{R})\co
\rho_{A}=(\mu_{A}\ot H)\co (A\ot (\rho_{A}\co \eta_{A})).$
\item[(b4)]$(A\ot \Pi_{H}^{L})\co \rho_{A}= ((\mu_{A}\co
c_{A,A})\ot H)\co (A\ot (\rho_{A}\co \eta_{A})).$
\item[(b5)]$(A\ot \overline{\Pi}_{H}^{R})\co \rho_{A}\co
\eta_{A}=\rho_{A}\co\eta_{A}.$
\item[(b6)]$ (A\ot \Pi_{H}^{L})\co \rho_{A}\co
\eta_{A}=\rho_{A}\co\eta_{A}.$
\end{itemize}

This definition is similar to the notion of right $H$-comodule monoid  in the weak Hopf algebra setting and the proof for the equivalence of (b1)-(b6) is the same.  

Note that, if $H$ is a Hopf quasigroup  and $(A, \rho_{A})$ is a unital magma, which is also a right
$H$-comodule, we will say that $A$ is a right $H$-comodule magma if it satisfies (\ref{chmagma}) and $\eta_{H}\ot \eta_{A}=\rho_{A}\co \eta_{A}$. In this case (b1)-(b6) trivialize.

}
\end{definition}

\begin{example}
\label{Hescomodulomagma}
{\rm Let $H$ be a weak Hopf quasigroup. Then $(H, \delta_H)$ is a right $H$-comodule magma.
}
\end{example}

\begin{definition}
\label{coinvariantesA}
{\rm  Let $H$ be a weak Hopf quasigroup and let $(A, \rho_A)$ be a right $H$-comodule
magma. We denote by $A^{co H}$ the equalizer of the morphisms $\rho_A$ and $(A\ot \Pi_{H}^{L})\co \rho_A$
(equivalently, $\rho_A$ and $(A\ot \overline{\Pi}_{H}^{R})\co \rho_A$) and by $i_A$ the injection of $A^{co H}$ in $A$.

The triple $(A^{co H}, \eta_{A^{co H}}, \mu_{A^{co H}})$ is a unital magma (the  submagma of coinvariants of $A$), where $\eta_{A^{co H}}:K\rightarrow A^{co H}$, $\mu_{A^{co H}}:A^{co H}\ot A^{co H}\rightarrow A^{co H}$ are the factorizations of the morphisms $\eta_A$ and $\mu_A\co (i_A\ot i_A)$ through $i_{A}$, respectively. Indeed, by  (b6) of Definition \ref{H-comodulomagma} we have that  $ (A\ot \Pi_{H}^{L})\co \rho_{A}\co
\eta_{A}=\rho_{A}\co\eta_{A}.$ As a consequence, there exists a unique morphism $\eta_{A^{co H}}:K\rightarrow A^{co H}$ such that 
\begin{equation}
\label{eta-coinv}
\eta_{A}=i_{A}\co \eta_{A^{co H}}.
\end{equation}
On the other hand, using (\ref{chmagma}), (b6) of Definition \ref{H-comodulomagma}  and (\ref{pi-mu-pi-pi}) we obtain 
\begin{itemize}
\item[ ]$\hspace{0.38cm} \rho_A\co \mu_A\co (i_A\ot i_A)$
\item[ ]$=\mu_{A\ot H}\co ((\rho_A\co i_A)\ot (\rho_A\co i_A))$
\item[ ]$=(\mu_{A}\ot (\mu_{H}\co (\Pi_{H}^{L}\ot \Pi_{H}^{L})))\co (A\ot c_{H,A}\ot H)\co ((\rho_A\co i_A)\ot (\rho_A\co i_A))$
\item[ ]$=(\mu_{A}\ot (\Pi_{H}^{L}\co \mu_{H}))\co (A\ot c_{H,A}\ot H)\co ((\rho_A\co i_A)\ot (\rho_A\co i_A))$
\item[ ]$=(A\ot \Pi_{H}^{L})\co \rho_A\co \mu_A\co (i_A\ot i_A).$
\end{itemize}
Therefore,  there exists a unique morphism  $\mu_{A^{co H}}:A^{co H}\ot A^{co H}\rightarrow A^{co H}$  satisfying 
\begin{equation}
\label{mu-coinv}
\mu_A\co (i_A\ot i_A)=i_{A}\co \mu_{A^{co H}}.
\end{equation}

}
\end{definition}

\begin{lemma}
\label{igualdadesmurho}
Let $H$ be a weak Hopf quasigroup and let $(A, \rho_A)$ be a right $H$-comodule magma. The following equalities hold:
\begin{equation}
\label{muArhoA-1}
\rho_A\co \mu_A\co (i_A\ot A)=(\mu_A\ot H)\co (i_A\ot \rho_A),
\end{equation}
\begin{equation}
\label{muArhoA-2}
\rho_A\co \mu_A\co (A\ot i_A)=(\mu_A\ot H)\co (A\ot c_{H,A})\co (\rho_A\ot i_A),
\end{equation}
\begin{equation}
\label{muArhoA-22}
(\mu_{A}\ot (\mu_{H}\co (H\ot \Pi_{H}^{L})))\co (A\ot c_{H,A}\ot H)\co (\rho_{A}\ot \rho_{A})=(\mu_A\ot H)\co (A\ot c_{H,A})\co (\rho_A\ot A).
\end{equation}

\end{lemma}

\begin{proof}
The first equality follows because $A$ is a right $H$-comodule magma, the properties of 
the equalizer $i_A$, (\ref{mu-pi-r-var}) and the naturalness of $c$. Indeed,

\begin{itemize}

\item[ ]$\hspace{0.38cm} \rho_A\co \mu_A\co (i_A\ot A)$
\item[ ]$=\mu_{A\ot H}\co ((\rho_A\co i_A)\ot \rho_A)$
\item[ ]$=\mu_{A\ot H}\co (((A\ot \overline{\Pi}_{H}^{R})\co\rho_A\co i_A)\ot \rho_A)$
\item[ ]$=(\mu_A\ot (\varepsilon_H\co \mu_H)\ot H)\co (A\ot c_{H,A}\ot \delta_H)\co ((\rho_A\co i_A)\ot \rho_A)$
\item[ ]$=(((A\ot \varepsilon_H)\co \rho_A)\ot H)\co (\mu_A\ot H)\co (i_A\ot \rho_A)$
\item[ ]$=(\mu_A\ot H)\co (i_A\ot \rho_A).$

\end{itemize}

In a similar way, but using (\ref{mu-pi-l}), we get (\ref{muArhoA-2}):
\begin{itemize}
\item[ ]$\hspace{0.38cm} \rho_A\co \mu_A\co (A\ot i_A)$
\item[ ]$=\mu_{A\ot H}\co (\rho_A\ot (\rho_A\co i_A))$
\item[ ]$=\mu_{A\ot H}\co (\rho_A\ot ((A\ot \Pi_{H}^{L})\co\rho_A\co i_A))$
\item[ ]$=(\mu_A\ot (((\varepsilon_H\co \mu_H)\ot H)\co (H\ot c_{H,H})\co (\delta_H\ot H)))\co (A\ot c_{H,A}\ot H)\co (\rho_A\ot (\rho_A\co i_A))$
\item[ ]$=(((A\ot \varepsilon_H)\co \rho_A)\ot H)\co (\mu_A\ot H)\co (A\ot c_{H,A})\co (\rho_A\ot i_A)$
\item[ ]$=(\mu_A\ot H)\co (A\ot c_{H,A})\co (\rho_A\ot i_A).$
\end{itemize}

Finally, 
\begin{itemize}
\item[ ]$\hspace{0.38cm} (\mu_{A}\ot (\mu_{H}\co (H\ot \Pi_{H}^{L})))\co (A\ot c_{H,A}\ot H)\co (\rho_{A}\ot \rho_{A})$
\item[ ]$=  (\mu_{A}\ot (((\varepsilon_{H}\co \mu_{H})\ot H)\co (H\ot c_{H,H})\co (\delta_{H}\ot H)))\co (A\ot c_{H,A}\ot H)\co (\rho_{A}\ot \rho_{A}) $
\item[ ]$= (((A\ot \varepsilon_{H})\co \mu_{A\ot H}\co (\rho_{A}\ot \rho_{A}))\ot H)\co (A\ot c_{H,A})\co (\rho_{A}\ot A) $
\item[ ]$= (((A\ot \varepsilon_{H})\co \rho_{A}\co \mu_{A})\ot H)\co (A\ot c_{H,A})\co (\rho_{A}\ot A) $
\item[ ]$= (\mu_A\ot H)\co (A\ot c_{H,A})\co (\rho_A\ot A), $
\end{itemize}
where the first equality follows by (\ref{mu-pi-l}), the second one follows by the comodule condition of $A$ and  the naturalness of $c$, the third  one is a consequence of (\ref{chmagma}) and the last one relies on the counit properties. 
Therefore,  (\ref{muArhoA-22}) holds and the proof is complete.
\end{proof}

\begin{remark}
{\rm It is not difficult to see that the coinvariant submagma $H^{co H}$ of the right $H$-comodule magma $(H, \delta_H)$ is  $H_L$. Moreover in this case the equations (\ref{muArhoA-1}) and (\ref{muArhoA-2}) are (\ref{aux-1-monoid-hl}) and (\ref{aux-2-monoid-hl}) respectively.
}
\end{remark}

\begin{proposition}
\label{idempotentenabla}
Let $H$ be a weak Hopf quasigroup and let $(A, \rho_A)$ be a right $H$-comodule magma. The morphism $\nabla_A:A\ot H\rightarrow A\ot H$, defined as 
$$\nabla_A=\mu_{A\ot H}\co (A\ot H\ot (\rho_A\co \eta_A)),$$
is idempotent and it is a right $H$-comodule  morphism for $\rho_{A\ot H}=A\ot \delta_H$. Moreover, if $(A, \rho_A)$ is a right $H$-comodule magma, it satisfies that
\begin{equation}
\label{nablaAmodulo} 
\nabla_A\co (\mu_A\ot H)=(\mu_A\ot H)\co (A\ot \nabla_A).
\end{equation}
As a consequence, there exist an object $A\square H$ and morphisms $i_{A\ot H}$ and $p_{A\ot H}$ such that $\nabla_A=i_{A\ot H}\co p_{A\ot H}$ and $id_{A\square H}=p_{A\ot H}\co i_{A\ot H}$.
\end{proposition}

\begin{proof}
Note that, by (b3) of Definition \ref{H-comodulomagma}, we obtain that 
\begin{equation}
\label{new-exp-nabla}
\nabla_A=(A\ot (\mu_H\co c_{H,H}))\co (((A\ot \overline{\Pi}_{H}^{R})\co \rho_A)\ot H).
\end{equation} 

Then $\nabla_A$ is an idempotent morphism. Indeed:

\begin{itemize}

\item[ ]$\hspace{0.38cm} \nabla_A\co \nabla_A$
\item[ ]$=(A\ot (\mu_H\co (\mu_H\ot H)\co (H\ot (c_{H,H}\co (\overline{\Pi}_{H}^{R}\ot \overline{\Pi}_{H}^{R})\co \delta_H))))\co (A\ot c_{H,H})\co (\rho_A\ot H)$
\item[ ]$=(A\ot (\mu_H\co (\mu_H\ot \overline{\Pi}_{H}^{R})\co (H\ot c_{H,H})))\co (A\ot H\ot ((\mu_H\ot H)\co (H\ot (\delta_H\co \eta_H))))$
\item[ ]$\hspace{0.38cm} \co (A\ot c_{H,H})\co (\rho_A\ot H)$
\item[ ]$=(A\ot (\mu_H\co c_{H,H}))\co (A\ot (\varepsilon_H\co \mu_H\co (\mu_H\ot H))\ot H\ot H)\co (A\ot H\ot H\ot 
(\delta_{H}\co \eta_{H})\ot H)$
\item[ ]$\hspace{0.38cm}\co (\rho_A\ot ((\Pi_{H}^{R}\ot H)\co \delta_H))$

\item[ ]$=(A\ot (\mu_H\co c_{H,H}))\co (A\ot (\varepsilon_H\co \mu_H\co (H\ot \mu_H))\ot H\ot H)\co  (A\ot H\ot H\ot 
(\delta_{H}\co \eta_{H})\ot H)$
\item[ ]$\hspace{0.38cm}\co (\rho_A\ot ((\Pi_{H}^{R}\ot H)\co \delta_H))$

\item[ ]$=(A\ot (\varepsilon_H\co \mu_H)\ot H)\co (\rho_A\ot (\mu_{H\ot H}\co (((\Pi_{H}^{R}\ot H)\co \delta_H)\ot (\delta_H\co \eta_H))))$

\item[ ]$=(A\ot (\varepsilon_H\co \mu_H)\ot H)\co (\rho_A\ot (\Pi_{H}^{R}*\Pi_{H}^{R})\ot H)\co (A\ot \delta_H)$

\item[ ]$=(A\ot (\varepsilon_H\co \mu_H)\ot H)\co (\rho_A\ot \Pi_{H}^{R}\ot H)\co (A\ot \delta_H)$

\item[ ]$=(A\ot \varepsilon_H\ot H)\co (A\ot \mu_{H\ot H})\co (\rho_A\ot H\ot (\delta_{H}\co \eta_H))$

\item[ ]$=\nabla_A.$

\end{itemize}

In the preceding computations, the first equality follows by (\ref{new-exp-nabla}), the naturalness of $c$ and because $A$ is a right $H$-comodule; the second one by (\ref{delta-pi-r-var}) and by the naturalness of $c$. In the third one we use (\ref{delta-pi-r}), the naturalness of $c$ and the definiton of $\overline{\Pi}_{H}^{R}$; the fourth one relies on (a2) of Definition \ref{Weak-Hopf-quasigroup}; the fifth one on the naturalness of $c$; the sixth one on the coassociativity of the coproduct and on (\ref{delta-pi-r}). The seventh equality is a consequence of (a4-7) and (a4-3) of Definition \ref{Weak-Hopf-quasigroup}, the eighth one follows by (\ref{delta-pi-r}) and finally, the last one follows by the naturalness of $c$, the definiton of $\overline{\Pi}_{H}^{R}$ and (\ref{new-exp-nabla}).

Now, using (a1) of Definition \ref{Weak-Hopf-quasigroup}, the condition of right $H$-comodule for $A$ and (b6) of Definition \ref{H-comodulomagma}, and the naturalness of $c$ and (\ref{2-mu-delta-pi-l}), we get that $\nabla_A$ is a right $H$-comodule morphism, i.e. 
\begin{equation}
\label{nabla-comod}
(A\ot \delta_H)\co \nabla_A=(\nabla_A\ot H)\co (A\ot \delta_H).
\end{equation}
Indeed,
\begin{itemize}
\item[ ]$\hspace{0.38cm} (A\ot \delta_H)\co \nabla_A$
\item[ ]$=(\mu_A\ot \mu_{H\ot H})\co (A\ot A\ot \delta_H\ot \delta_H)\co (A\ot c_{H,A}\ot H)\co (A\ot H\ot (\rho_A\co \eta_A))$
\item[ ]$=(\mu_A\ot \mu_{H\ot H})\co (A\ot A\ot \delta_H\ot ((H\ot \Pi_{H}^{L})\co \delta_H))\co (A\ot c_{H,A}\ot H)\co (A\ot H\ot (\rho_A\co \eta_A))$
\item[ ]$=(\nabla_A\ot H)\co (A\ot \delta_H).$

\end{itemize}

Finally, 
\begin{itemize}

\item[ ]$\hspace{0.38cm} \nabla_A\co (\mu_A\ot H)$
\item[ ]$=(\mu_A\ot (\varepsilon_H\co \mu_H\co (\mu_H\ot H))\ot (\mu_H\co c_{H,H}))\co (((A\ot c_{H,A}\ot H)\co (\rho_A\ot \rho_A))\ot (\delta_H\co \eta_H)\ot H)$
\item[ ]$=(\mu_A\ot (\varepsilon_H\co \mu_H\co (H\ot \mu_H))\ot (\mu_H\co c_{H,H}))\co (((A\ot c_{H,A}\ot H)\co (\rho_A\ot \rho_A))\ot (\delta_H\co \eta_H)\ot H)$
\item[ ]$=(\mu_A\ot (\varepsilon_H\co \mu_H)\ot (\mu_H\co c_{H,H}))\co (A\ot c_{H,A}\ot ((H\ot \overline{\Pi}_{H}^{R})\co \delta_H)\ot H)\co (\rho_A\ot \rho_A\ot H)$
\item[ ]$= (A\ot \varepsilon_{H}\ot (\mu_H\co c_{H,H}\co (\overline{\Pi}_{H}^{R}\ot H)))\co ((\mu_{A\ot H}\co (\rho_{A}\ot \rho_{A}))\ot H\ot H)\co (A\ot \rho_{A}\ot H)$
\item[ ]$=(((A\ot \varepsilon_H)\co \rho_A\co \mu_A)\ot H)\co (A\ot \nabla_A)$
\item[ ]$=(\mu_A\ot H)\co (A\ot \nabla_A),$

\end{itemize}
where the first and fifth equalities follow by (\ref{chmagma}) and (\ref{new-exp-nabla}), the second  one by (a2) of Definition \ref{Weak-Hopf-quasigroup} and the third one by (\ref{delta-pi-r-var}). In the fourth equality we used that $A$ is a right $H$-comodule,  and the last one follows by the counit properties.

Therefore, (\ref{nablaAmodulo}) holds and the proof is complete. 

\end{proof}

Note that, by the lack of associativity, for $M=A\ot H$, $\varphi_{M}=\mu_A\ot H$ is not a left $A$-module structure (i.e. $\varphi_{M}\co (\eta_{A}\ot M)=id_{M}$, $\varphi_{M}\co (A\ot \varphi_{M})=\varphi_{M}\co (\mu_{A}\ot M)$). Moreover, if $A=H$, by 
(\ref{delta-pi-r}), we have 
\begin{equation}
\label{nabladeH}
\nabla_H=(\mu_H \ot H)\co (H\ot  \Pi_{H}^{R}\ot H)\co (H\ot \delta_H).
\end{equation}

\begin{lemma}
\label{igualdadesnabla}
Let $H$ be a weak Hopf quasigroup and let $(A, \rho_A)$ be a right $H$-comodule magma. The following equalities hold:
\begin{equation}
\label{nabla-1}
p_{A\ot H}\co (A\ot \mu_H)\co (c_{H,A}\ot H)\co (H\ot (\rho_A\co \eta_A))=p_{A\ot H}\co(\eta_A\ot H),
\end{equation}
\begin{equation}
\label{nabla-2}
 (A\ot (\delta_H\co \mu_H))\co (c_{H,A}\ot H)\co (H\ot (\rho_A\co \eta_A))=
(((A\ot \mu_H)\co (c_{H,A}\ot H)\co (H\ot (\rho_A\co \eta_A)))\ot H)\co \delta_H,
\end{equation}
\begin{equation}
\label{nabla-3}
\nabla_A\co (\mu_A\ot H)\co (A\ot \rho_A)=(\mu_A\ot H)\co (A\ot \rho_A).
\end{equation}

\end{lemma}

\begin{proof}

The equality (\ref{nabla-1}) holds because, composing with $i_{A\ot H}$, we have 
\begin{itemize}
\item[ ]$\hspace{0.38cm} \nabla_A\co (A\ot \mu_{H})\co (c_{H,A}\ot H)\co (H\ot (\rho_{A}\co \eta_{A}))$
\item[ ]$= (A\ot (\mu_{H}\co (\mu_{H}\ot H)))\co (c_{H,A}\ot H\ot H)\co (H\ot \mu_{A}\ot H\ot H)\co (H\ot A\ot c_{H,A}\ot H)$
\item[ ]$\hspace{0.38cm} \co (H\ot (\rho_{A}\co \eta_{A})\ot (\rho_{A}\co \eta_{A}))$
\item[ ]$= (A\ot (\mu_{H}\co ((\mu_{H}\co (H\ot \Pi_{H}^{L}))\ot H)))\co (c_{H,A}\ot H\ot H)\co (H\ot \mu_{A}\ot H\ot H)\co (H\ot A\ot c_{H,A}\ot H)$
\item[ ]$\hspace{0.38cm} \co (H\ot (\rho_{A}\co \eta_{A})\ot (\rho_{A}\co \eta_{A}))$
\item[ ]$= (A\ot (\mu_{H}\co (H\ot (\mu_{H}\co (\Pi_{H}^{L}\ot H)))))\co (c_{H,A}\ot H\ot H)\co (H\ot \mu_{A}\ot H\ot H)\co (H\ot A\ot c_{H,A}\ot H)$
\item[ ]$\hspace{0.38cm} \co (H\ot (\rho_{A}\co \eta_{A})\ot (\rho_{A}\co \eta_{A}))$
\item[ ]$= (A\ot \mu_{H})\co (c_{H,A}\ot H)\co (H\ot (\mu_{A\ot H}\co ((\rho_{A}\co \eta_{A})\ot (\rho_{A}\co \eta_{A})))) $
\item[ ]$=(A\ot \mu_{H})\co (c_{H,A}\ot H)\co (H\ot (\rho_{A}\co \mu_{A}\co (\eta_{A}\ot \eta_{A})))$
\item[ ]$=(A\ot \mu_{H})\co (c_{H,A}\ot H)\co (H\ot  (\rho_{A}\co \eta_{A}))$
\item[ ]$=\nabla_A\co (\eta_{A}\ot H),   $
\end{itemize}
where the first equality follows by the naturalness of $c$, the second one follows by (b6) of Definition \ref{H-comodulomagma}, and the third one follows by (\ref{monoid-hl-2}) and by the naturalness of $c$. In the fourth equality  we used the naturalness of $c$ and (b6) of Definition \ref{H-comodulomagma}. The fifth equality is a consequence of (\ref{chmagma}) and the sixth and seventh ones rely on the properties of the unit of $A$. 

On the other hand, the proof for (\ref{nabla-2}) is the following:
\begin{itemize}
\item[ ]$\hspace{0.38cm} (A\ot (\delta_H\co \mu_H))\co  (c_{H,A}\ot H)\co (H\ot (\rho_{A}\co \eta_{A}))$
\item[ ]$= (A\ot ((\mu_{H}\ot \mu_{H})\co \delta_{H\ot H}))\co (c_{H,A}\ot H)\co (H\ot (\rho_{A}\co \eta_{A})) $
\item[ ]$=  (A\ot (\mu_{H\ot H}\co (\delta_{H}\ot H\ot \Pi_{L}^{H})))\co (c_{H,A}\ot H\ot H)\co (H\ot ((\rho_{A}\ot H)\co \rho_{A}\co \eta_{A}))$
\item[ ]$= (A\ot ((\mu_{H}\ot (\mu_{H}\co (H\ot \Pi_{H}^{L})))\co \delta_{H\ot H}))\co (c_{H,A}\ot H)\co (H\ot (\rho_{A}\co \eta_{A})) $
\item[ ]$=  (A\ot ((\mu_{H}\ot H)\co (H\ot c_{H,H})\co (\delta_{H}\ot H)))\co (c_{H,A}\ot H)\co (H\ot (\rho_{A}\co \eta_{A}))$
\item[ ]$= (((A\ot \mu_H)\co (c_{H,A}\ot H)\co (H\ot (\rho_A\co \eta_A)))\ot H)\co \delta_H.$
\end{itemize}
In these equalities the first one is consequence of (a1) of Definition (\ref{Weak-Hopf-quasigroup}), the second one holds because $A$ is a right $H$-comodule and by (b6) of Definition \ref{H-comodulomagma}. In the third one we applied again that $A$ is a right $H$-comodule, the fourth one follows by (\ref{2-mu-delta-pi-l}) and the last one relies on the naturalness of $c$.

Finally,  (\ref{nabla-3}) is a direct consequence of the equalities (\ref{nablaAmodulo}) and 
\begin{equation}
\label{rho-nabla}
\nabla_A\co \rho_A=\rho_A.
\end{equation}
Note that (\ref{rho-nabla}) holds because, by (\ref{chmagma}) and the unit properties, we have 
$$\nabla_A\co \rho_A =\mu_{A\ot H}\co (\rho_A\ot (\rho_A\co \eta_A))=\rho_A\co \mu_{A}\co (A\ot \eta_{A})=\rho_{A}.$$

\end{proof}

\begin{proposition}
\label{monoidecoinvariantes}
Let $H$ be a weak Hopf quasigroup and let $(A, \rho_A)$ be a right $H$-comodule magma such that 
\begin{equation}
\label{AsubH-2}
\mu_{A}\co (A\ot (\mu_A\co (i_A\ot A)))=\mu_A\co ((\mu_A\co (A\ot i_A))\ot A).
\end{equation}
Then $(A^{co H}, \eta_{A^{co H}}, \mu_{A^{co H}})$ is a monoid.
Moreover the morphism 
$$\overline{\gamma}_A=p_{A\ot H}\co (\mu_A\ot H)\co (A\ot \rho_A):A\ot A\rightarrow A\square H$$ factorizes through the coequalizer diagram

$$
\setlength{\unitlength}{1mm}
\begin{picture}(101.00,10.00)
\put(22.00,8.00){\vector(1,0){40.00}}
\put(22.00,4.00){\vector(1,0){40.00}}
\put(75.00,6.00){\vector(1,0){21.00}}
\put(43.00,11.00){\makebox(0,0)[cc]{$(\mu_{A}\co (A\ot i_A))\ot A$ }}
\put(43.00,0.00){\makebox(0,0)[cc]{$A\ot (\mu_{A}\co (i_A\ot A))$ }}
\put(85.00,9.00){\makebox(0,0)[cc]{$n_{A}$ }}
\put(10.00,6.00){\makebox(0,0)[cc]{$ A\ot A^{co H}\ot A$ }}
\put(70.00,6.00){\makebox(0,0)[cc]{$A\ot A$ }}
\put(105.00,6.00){\makebox(0,0)[cc]{$A\ot_{A^{co H}}A$ }}
\end{picture}
$$
and, if we denote by $\gamma_A$ this factorization, the following equalities:
\begin{equation}
\label{gammarho-1}
(\gamma_A\ot H)\co \rho^{1}_{A\ot_{A^{co H}}A}=(p_{A\ot H}\ot H)\co (A\ot c_{H,H})\co (A\ot (\mu_{H}\co (H\ot \lambda_{H}))\ot H)\co (\rho_A\ot \delta_{H})\co i_{A\ot H}\co \gamma_A,
\end{equation}
\begin{equation}
\label{gammarho-2}
(\gamma_A\ot H)\co \rho^{2}_{A\ot_{A^{co H}}A}=(p_{A\ot H}\ot H)\co(A\ot \delta_H)\co i_{A\ot H}\co \gamma_A,
\end{equation}
hold, where $\rho^{1}_{A\ot_{A^{co H}}A}$ and $\rho^{2}_{A\ot_{A^{co H}}A}$ are the factorizations, through the coequalizer $n_{A}$, of the morphisms 
$(n_{A}\ot H)\co (A\ot c_{H,A})\co (\rho_A\ot A)$ and $(n_{A}\ot H)\co (A\ot \rho_A)$, respectively.

\end{proposition}

\begin{proof} Trivially, if (\ref{AsubH-2}) 	holds, the triple $(A^{co H}, \eta_{A^{co H}}, \mu_{A^{co H}})$ is a monoid. On the other hand, consider the coequalizer diagram 

$$
\setlength{\unitlength}{1mm}
\begin{picture}(101.00,10.00)
\put(22.00,8.00){\vector(1,0){40.00}}
\put(22.00,4.00){\vector(1,0){40.00}}
\put(75.00,6.00){\vector(1,0){21.00}}
\put(43.00,11.00){\makebox(0,0)[cc]{$(\mu_{A}\co (A\ot i_A))\ot A$ }}
\put(43.00,0.00){\makebox(0,0)[cc]{$A\ot (\mu_{A}\co (i_A\ot A))$ }}
\put(85.00,9.00){\makebox(0,0)[cc]{$n_{A}$ }}
\put(10.00,6.00){\makebox(0,0)[cc]{$ A\ot A^{co H}\ot A$ }}
\put(70.00,6.00){\makebox(0,0)[cc]{$A\ot A$ }}
\put(105.00,6.00){\makebox(0,0)[cc]{$A\ot_{A^{co H}}A$ }}
\end{picture}
$$
By (\ref{muArhoA-1}) and (\ref{AsubH-2}) we have 
$$(\mu_{A}\ot H)\co (A\ot \rho_{A})\co (A\ot (\mu_{A}\co (i_A\ot A)))=((\mu_{A}\co (A\ot \mu_{A}))\ot H)\co (A\ot i_{A}\ot \rho_{A})=(\mu_{A}\ot H)\co ((\mu_{A}\co (A\ot i_{A}))\ot \rho_{A})$$
and, therefore, there exists a unique morphism such that 
\begin{equation}
\label{can-fact}
\gamma_{A}\co n_{A}=\overline{\gamma}_A.
\end{equation}
Also, by (\ref{muArhoA-1}), (\ref{muArhoA-2}), the naturalness of $c$, and the definition of $n_{A}$,  we have  
$$(n_{A}\ot H)\co (A\ot c_{H,A})\co (\rho_A\ot A) \co ((\mu_{A}\co (A\ot i_A))\ot A)=(n_{A}\ot H)\co (A\ot c_{H,A})\co (\rho_A\ot A) \co (A\ot (\mu_{A}\co (i_A\ot A)))$$
and 
$$(n_{A}\ot H)\co (A\ot \rho_A)\co ((\mu_{A}\co (A\ot i_A))\ot A)=(n_{A}\ot H)\co (A\ot \rho_A) \co (A\ot (\mu_{A}\co (i_A\ot A))).$$ 
Then, there exists unique morphisms $\rho^{1}_{A\ot_{A^{co H}}A}, \rho^{2}_{A\ot_{A^{co H}}A}: A\ot_{A^{co H}}A\rightarrow A\ot_{A^{co H}}A\ot H$ such that 
\begin{equation}
\label{rho-fact-1}
\rho^{1}_{A\ot_{A^{co H}}A}\co n_{A}=(n_{A}\ot H)\co (A\ot c_{H,A})\co (\rho_A\ot A), 
\end{equation}
\begin{equation}
\label{rho-fact-2}
\rho^{2}_{A\ot_{A^{co H}}A}\co n_{A}=(n_{A}\ot H)\co (A\ot \rho_A),
\end{equation}
respectively. 

For $\rho^{1}_{A\ot_{A^{co H}}A}$ the equality (\ref{gammarho-1}) holds because by composing with the coequalizer $n_{A}$,
\begin{itemize}
\item[ ]$\hspace{0.38cm} (p_{A\ot H}\ot H)\co (A\ot c_{H,H})\co (A\ot (\mu_{H}\co (H\ot \lambda_{H}))\ot H)\co (\rho_A\ot \delta_{H})\co i_{A\ot H}\co \gamma_A\co n_{A}$
\item[ ]$= (p_{A\ot H}\ot H)\co (A\ot c_{H,H})\co (A\ot (\mu_{H}\co (H\ot \lambda_{H}))\ot H)\co (\rho_A\ot \delta_{H})\co \nabla_{A}\co (\mu_{A}\ot H)\co (A\ot \rho_{A}) $
\item[ ]$= (p_{A\ot H}\ot H)\co (A\ot c_{H,H})\co (A\ot (\mu_{H}\co (H\ot \lambda_{H}))\ot H)\co ((\rho_A\co \mu_{A})\ot \delta_{H}) \co (A\ot \rho_{A}) $
\item[ ]$= (p_{A\ot H}\ot H)\co (A\ot c_{H,H})\co (A\ot (\mu_{H}\co (H\ot \lambda_{H}))\ot H)\co (( \mu_{A\ot H}\co (\rho_{A}\ot \rho_{A}))\ot \delta_{H})\co  (A\ot \rho_{A}) $
\item[ ]$= (p_{A\ot H}\ot H)\co (A\ot c_{H,H})\co (\mu_{A}\ot (\mu_{H}\co (\mu_{H}\ot \lambda_{H})\co (H\ot \delta_{H}))\ot H)\co (A\ot c_{H,A}\ot \delta_{H})\co (\rho_{A}\ot \rho_{A})  $
\item[ ]$= (p_{A\ot H}\ot H)\co (A\ot c_{H,H})\co (\mu_{A}\ot (\mu_{H}\co (H\ot \Pi_{H}^{L}))\ot H)\co (A\ot c_{H,A}\ot \delta_{H})\co (\rho_{A}\ot \rho_{A})  $
\item[ ]$= (p_{A\ot H}\ot H)\co (A\ot c_{H,H})\co (\mu_{A}\ot (((\varepsilon_{H}\co \mu_{H})\ot H)\co (H\ot c_{H,H})\co (\delta_{H}\ot H))\ot H)\co (A\ot c_{H,A}\ot \delta_{H})$
\item[ ]$\hspace{0.38cm}\co (\rho_{A}\ot \rho_{A})  $
\item[ ]$= (p_{A\ot H}\ot H)\co (((A\ot \varepsilon_{H})\co \mu_{A\ot H}\co (\rho_{A}\ot \rho_{A}))\ot c_{H,H})\co (A\ot c_{H,A}\ot H)\co (\rho_{A}\ot \rho_{A})$
\item[ ]$= (p_{A\ot H}\ot H)\co (((A\ot \varepsilon_{H})\co \rho_A\co \mu_{A})\ot c_{H,H})\co (A\ot c_{H,A}\ot H)\co (\rho_{A}\ot \rho_{A}) $
\item[ ]$=(\overline{\gamma}_A\ot H)\co (A\ot c_{H,A})\co (\rho_{A}\ot A)  $
\item[ ]$=((\gamma_A\co n_{A})\ot H)\co (A\ot c_{H,A})\co (\rho_{A}\ot A)  $
\item[ ]$=(\gamma_A\ot H)\co \rho^{1}_{A\ot_{A^{co H}}A}\co n_{A},$
\end{itemize}
where the first and the tenth equalities follow by (\ref{can-fact}), the second one follows by (\ref{nabla-3}) and the third and eighth ones follow by (\ref{chmagma}). In the fourth identity we used that $A$ is a right $H$-comodule and the coassociativity of $\delta_{H}$. The fifth equality  relies on (a4-6) of Definition \ref{Weak-Hopf-quasigroup} and the sixth one is a consequence of (\ref{mu-pi-l}). In the seventh equality we applied the naturalness of $c$ and the comodule structure of $A$, the ninth one follows by the counit properties and the naturalness of $c$ and the last one follows by (\ref{rho-fact-1}). 

Finally, by (\ref{rho-fact-2}), the comodule structure of $A$ and (\ref{nabla-comod}) we have 
$$(\gamma_A\ot H)\co \rho^{2}_{A\ot_{A^{co H}}A}\co n_{A}=(p_{A\ot H}\ot H)\co(A\ot \delta_H)\co i_{A\ot H}\co \gamma_A\co n_{A},$$
and then (\ref{gammarho-2}) holds.

\end{proof}

\begin{lemma}
\label{morfismosauxiliares}
Let $H$ be a weak Hopf quasigroup and let $(A, \rho_A)$ be a right $H$-comodule magma such that the functor $A\ot -$ preserves coequalizers. Assume that 
\begin{equation}
\label{AsubH-3}
\mu_{A}\co (A\ot (\mu_A\co (A\ot i_A)))=\mu_A\co (\mu_A\ot i_A)).
\end{equation}
Then the morphism $n_{A}\co (\mu_A\ot A)$ factorizes though the coequalizer $A\ot n_{A}$. We will denote by $\varphi_{A\ot_{A^{co H}}A}$ this factorization, i.e., the unique morphism such that
\begin{equation}
\label{varphi}
\varphi_{A\ot_{A^{co H}}A}\co (A\ot n_{A})=n_{A}\co (\mu_A\ot A).
\end{equation}

\end{lemma}

\begin{proof} If the functor $A\ot -$ preserves coequalizers, we have that 

$$
\setlength{\unitlength}{1mm}
\begin{picture}(120.00,10.00)
\put(19.00,8.00){\vector(1,0){40.00}}
\put(19.00,4.00){\vector(1,0){40.00}}
\put(80.00,6.00){\vector(1,0){21.00}}
\put(41.00,11.00){\makebox(0,0)[cc]{$A\ot (\mu_{A}\co (A\ot i_A))\ot A$ }}
\put(41.00,0.00){\makebox(0,0)[cc]{$A\ot A\ot (\mu_{A}\co (i_A\ot A))$ }}
\put(88.00,9.00){\makebox(0,0)[cc]{$A\ot n_{A}$ }}
\put(3.00,6.00){\makebox(0,0)[cc]{$ A\ot A\ot A^{co H}\ot A$ }}
\put(71.00,6.00){\makebox(0,0)[cc]{$A\ot A\ot A$ }}
\put(118.00,6.00){\makebox(0,0)[cc]{$A\ot A\ot_{A^{co H}}A$ }}
\end{picture}
$$
is a coequalizer diagram, and then the result follows easily by (\ref{AsubH-3}) and by the properties of $n_{A}$.
\end{proof}

Now we introduce the definition of Galois extension associated to a weak Hopf quasigroup.

\begin{definition}
\label{Galois}
{\rm Let $H$ be a weak Hopf quasigroup and let $(A, \rho_A)$ be a right $H$-comodule magma satisfying (\ref{AsubH-2}). We say that $A^{co H}\hookrightarrow A$ is a weak $H$-Galois extension if the morphism $\gamma_A$ is an isomorphism.  

Let $\rho^{2}_{A\ot_{A^{co H}}A}$ be the morphism introduced in  Proposition \ref{monoidecoinvariantes}. The pair  $(A\ot_{A^{co H}}A, \rho^{2}_{A\ot_{A^{co H}}A})$ is a right $H$-comodule  and so is $(A\square H, \rho_{A\square H})$ with 
$$\rho_{A\square H}=(p_{A\ot H}\ot H)\co (A\ot\delta_{H})\co i_{A\ot H}.$$
Then, $\gamma_A$ is a morphism of right $H$-comodules,  because composing with $n_{A}$ and using (\ref{can-fact}), (\ref{nabla-comod}) and (\ref{gammarho-2}), the equality 
$$\rho_{A\square H}\co \gamma_A\co n_{A}= (\gamma_A\ot H)\co \rho^{2}_{A\ot_{A^{co H}}A}\co n_{A}$$
holds and therefore 
\begin{equation}
\label{gamma-comod}
\rho_{A\square H}\co \gamma_A=(\gamma_A\ot H)\co \rho^{2}_{A\ot_{A^{co H}}A}.
\end{equation}
On the other hand, if $\varphi_{A\square H}= p_{A\ot H}\co (\mu_{A}\ot H)\co (A\ot i_{A\ot H})$, by (\ref{can-fact}) and (\ref{nablaAmodulo}),  we obtain that $\gamma_{A}$ is almost lineal, i.e., 
\begin{equation}
\label{gamma-almost-lineal}
\varphi_{A\square H} \co (A\ot (\gamma_A\co n_{A}\co (\eta_{A}\ot A)))=\gamma_A\co n_{A}.
\end{equation}

If $A^{co H}\hookrightarrow A$ is a  weak $H$-Galois extension such that the functor $A\ot -$ preserves coequalizers, and  the equality (\ref{AsubH-3}) holds, we will say that  $\gamma_A^{-1}$ is almost lineal if it satisfies that
\begin{equation}
\label{almostlineal}
\gamma_A^{-1}\co p_{A\ot H}=\varphi_{A\ot_{A^{co H}}A}\co (A\ot (\gamma_A^{-1}\co p_{A\ot H}\co (\eta_A\ot H))). 
\end{equation}

}
\end{definition}

\begin{definition}
\label{basenormal}
{\rm
Let $A^{co H}\hookrightarrow A$ be a weak $H$-Galois extension. We will say that $A^{co H}\hookrightarrow A$ is a weak $H$-Galois with normal basis if there exists an idempotent morphism of left $A^{co H}$-modules ($\varphi_{A^{co H}\ot H}=\mu_{A^{co H}}\ot H$) and right $H$-comodules ($\rho_{A^{co H}\ot H}=A^{co H}\ot \delta_H$), 
$$\Omega_A:A^{co H}\ot H\rightarrow A^{co H}\ot H,$$ 
and an isomorphism of left $A^{co H}$-modules and right $H$-comodules 
$$b_A:A\rightarrow A^{co H}\times H,$$ 
where $A^{co H}\times H$ is the image of $\Omega_A$ and 
$\varphi_{A^{co H}\times H}=r_A\co (\mu_{A^{co H}}\ot H)\co (A^{co H}\ot s_A)$, $\rho_{A^{co H}\times H}=(r_A\ot H)\co (A^{co H}\ot \delta_H)\co s_A$, being $s_A:A^{co H}\times H\rightarrow A^{co H}\ot H$ and $r_A:A^{co H}\ot H\rightarrow A^{co H}\times H$ the morphisms such that $s_A\co r_A=\Omega_A$ and $r_A\co s_A=id_{A^{co H}\times H}$.

Note that by Proposition \ref{monoidecoinvariantes}, $A^{co H}$ is a monoid and then $\varphi_{A^{co H}\ot H}$ is a left $A^{co H}$-module structure for $A^{co H}\ot H$.
}
\end{definition}

\begin{remark}
\label{Galoiswhaandhq}
{\rm In the weak Hopf algebra setting, Definition \ref{Galois} is a generalization of the notion of weak $H$-Galois extension (with normal basis) given in \cite{AFG2}.

Recall that if $H$ is a weak Hopf algebra and $A$ a right $H$-comodule monoid, the equality (\ref{almostlineal}) is always true. Indeed, by the definitions of $\varphi_{A\ot_{A^{co H}}A}$ and $\gamma_A$ and taking into account that $A$ is a monoid and (\ref{nabla-3}),

\begin{itemize}
\item[ ]$\hspace{0.38cm} \gamma_A\co \varphi_{A\ot_{A^{co H}}A}\co (A\ot n_{A})$
\item[ ]$=\gamma_A\co  n_{A}\co (\mu_A\ot A)$
\item[ ]$=p_{A\ot H}\co (\mu_A\ot H)\co (A\ot \rho_A)\co (\mu_A\ot A)$
\item[ ]$=p_{A\ot H}\co (\mu_A\ot H)\co (A\ot (\nabla_A\co (\mu_A\ot H)\co (A\ot \rho_A)))$
\item[ ]$=p_{A\ot H}\co (\mu_A\ot H)\co (A\ot (i_{A\ot H}\co \gamma_A\co n_{A})),$
\end{itemize}

and then $\gamma_A\co \varphi_{A\ot_{A^{co H}}A}=p_{A\ot H}\co (\mu_A\ot H)\co (A\ot (i_{A\ot H}\co \gamma_A))$. Therefore
\begin{itemize}
\item[ ]$\hspace{0.38cm} \varphi_{A\ot_{A^{co H}}A}\co (A\ot (\gamma_A^{-1}\co p_{A\ot H}\co (\eta_A\ot H)))$
\item[ ]$=\gamma_A^{-1}\co \gamma_A\co \varphi_{A\ot_{A^{co H}}A}\co (A\ot (\gamma_A^{-1}\co p_{A\ot H}\co (\eta_A\ot H)))$
\item[ ]$=\gamma_A^{-1}\co p_{A\ot H}\co (\mu_A\ot H)\co (A\ot (i_{A\ot H}\co \gamma_A\co \gamma_A^{-1}\co p_{A\ot H}\co (\eta_A\ot H)))$
\item[ ]$=\gamma_A^{-1}\co p_{A\ot H}\co (\mu_{A}\ot H)\co  (A\ot (\nabla_A\co (\eta_A\ot H)))$
\item[ ]$=\gamma_A^{-1}\co p_{A\ot H},$
\end{itemize}
and $\gamma_A^{-1}$ is almost lineal.

On the other hand, if $H$ is a Hopf quasigroup, $\nabla_A=id_{A\ot H}$ and then $\gamma_A$ is the factorization through the coequalizer of the morphism $(\mu_A\ot H)\co (A\ot \rho_A)$. Then, for this algebraic structure, Definition \ref{Galois} is the notion of $H$-Galois extension for Hopf quasigroups (see \cite{AFG-3}).
Also, $\varphi_{A\square H}=\mu_{A}\ot H$, and, as a consequence, the condition of almost lineal for $\gamma_{A}$ is 
\begin{equation}
\label{gamma-almostlinealquasigroup}
(\mu_{A}\ot H)\co (A\ot (\gamma_A\co n_{A}\co (\eta_{A}\ot A)))=\gamma_A\co n_{A}.
\end{equation}
Now condition almost lineal for $\gamma_A^{-1}$ says that the equality
\begin{equation}
\label{almostlinealquasigroup}
\gamma_A^{-1}=\varphi_{A\ot_{A^{co H}}A}\co (A\ot (\gamma_A^{-1}\co (\eta_A\ot H))) 
\end{equation}
holds.
}
\end{remark}

\begin{example}
\label{HesGalois}
{\rm Let $H$ be a weak Hopf quasigroup. Then $H_L\hookrightarrow H$ is a weak $H$-Galois extension with normal basis. Also, $\gamma_H^{-1}$ is almost lineal. 

First of all, note that by Proposition \ref{monoid-hl}, equalities (\ref{AsubH-2}) and (\ref{AsubH-3}) hold for the right $H$-comodule magma $(H, \delta_H)$. Moreover, let $\gamma_H^{-1}=n_{H}\co (\mu_H\ot H)\co (H\ot \lambda_H\ot H)\co (H\ot \delta_H)\co i_{H\ot H}:H\square H\rightarrow H\ot_{H_L}H$. Then

\begin{itemize}
\item[ ]$\hspace{0.38cm} \gamma_H\co \gamma_H^{-1}$
\item[ ]$=p_{H\ot H}\co (\mu_H\ot H)\co (H\ot \delta_H))\co (\mu_H\ot H)\co (H\ot \lambda_H\ot H)\co (H\ot \delta_H))\co i_{H\ot H}$
\item[ ]$=p_{H\ot H}\co (\mu_H\ot H)\co (H\ot \Pi_{H}^{R}\ot H)\co (H\ot \delta_H))\co i_{H\ot H}$
\item[ ]$=p_{H\ot H}\co \nabla_H\co i_{H\ot H}$
\item[ ]$=id_{H\square H}.$

\end{itemize}

In the preceding calculations, the first equality follows by the definition of $\gamma_H$; the second one relies on the coassociativity of $\delta_{H}$ and on (a4-7) of Definition \ref{Weak-Hopf-quasigroup}; in the third one we use (\ref{nabladeH}); finally, the last one is a direct consequence of the factorization of $\nabla_{H}$.
On the other hand,

\begin{itemize}
\item[ ]$\hspace{0.38cm} \gamma_H^{-1}\co \gamma_H\co n_H$
\item[ ]$=n_H\co  (\mu_H\ot H)\co (H\ot \lambda_H\ot H)\co (H\ot \delta_H)\co \nabla_H\co (\mu_H\ot H)\co (H\ot \delta_H)$
\item[ ]$=n_H\co  (\mu_H\ot H)\co (H\ot \lambda_H\ot H)\co (H\ot \delta_H)\co (\mu_H\ot H)\co (H\ot \delta_H)$
\item[ ]$=n_H\co (\mu_H\ot H)\co (H\ot \Pi_{H}^{L}\ot H)\co (H\ot \delta_H)$
\item[ ]$=n_H\co (\mu_H\ot H)\co (H\ot (i_L\co p_L)\ot H)\co (H\ot \delta_H)$
\item[ ]$=n_H\co (H\ot (\Pi_{H}^{L}*id_H))$
\item[ ]$=n_H,$
\end{itemize}
where the first equality follows by the definition of $\gamma_H$; the second one by applying (\ref{nabla-3}) to the right $H$-comodule magma $H$. The third equality is a consequence of the coassociativity of $\delta_{H}$ and (a4-6) of Definition \ref{Weak-Hopf-quasigroup}; the fourth one follows because $\Pi_{H}^{L}=i_L\co p_L$; the fifth equality uses the properties of  $n_{H}$ and the last one follows by (\ref{pi-l}). As a consequence, $\gamma_H^{-1}\co \gamma_H=id_{H\ot_{H_L}H}$ and $H_L\hookrightarrow H$ is a weak $H$-Galois extension.

Now we must show that the extension has a normal basis. Let $\Omega_H:H_L\ot H\rightarrow H_L\ot H$ be the morphism defined as 
$\Omega_H=(p_L\ot H)\co \delta_H\co \mu_H\co (i_L\ot H)$. By (\ref{pi-l}), $\Omega_H$ is idempotent. Moreover, using that $i_{L}$ is an equalizer, (a1) of Definition \ref{Weak-Hopf-quasigroup}, and (\ref{mu-pi-r-var}) we obtain that $\Omega_H=((p_L\co \mu_H)\ot H)\co (i_{L}\ot \delta_H)$ and then $\Omega_H$ is a right $H$-comodule morphism. Moreover, using (\ref{pi-delta-mu-pi-1})  and the equality (\ref{monoid-hl-2}),

\begin{itemize}
\item[ ]$\hspace{0.38cm} (\mu_{H_L}\ot H)\co (H_L\ot \Omega_H)$
\item[ ]$=((p_L\co \mu_H\co (i_{H_L}\ot \Pi_{H}^{L}))\ot H)\co (H_L\ot i_{H_L}\ot \delta_H)$
\item[ ]$=((p_L\co \mu_H\co (i_{H_L}\ot H))\ot H)\co (H_L\ot i_{H_L}\ot \delta_H)$
\item[ ]$=((p_L\co \mu_H\co (\mu_H\co (i_{H_L}\ot i_{H_L})\ot H))\ot H)\co (H_L\ot H_L\ot \delta_H)$
\item[ ]$=\Omega_H\co (\mu_{H_L}\ot H),$
\end{itemize}
and $\Omega_H$ is a morphism of left $H_L$-modules.
On the other hand, let $s_H:H_L\times H\rightarrow H_L\ot H$ and $r_H:H_L\ot H\rightarrow H_L\times H$ be the morphisms such that $s_H\co r_H=\Omega_H$ and $r_H\co s_H=id_{H_L\times H}$ and define $b_H=r_H\co (p_L\ot H)\co \delta_H$. It is not difficult to see that $b_H$ is a right $H$-comodule isomorphism with inverse $b_H^{-1}=\mu_H\co (i_{H_L}\ot H)\co s_H$. Moreover, 

\begin{itemize}
\item[ ]$\hspace{0.38cm} \varphi_{H_{L}\times H}\co (H_L\ot b_H)$
\item[ ]$=r_H\co (\mu_{H_L}\ot H)\co (H_L\ot \Omega_H)\co (H_L\ot ((p_L\ot H)\co \delta_H))$
\item[ ]$=r_H\co (\mu_{H_L}\ot H)\co (H_L\ot \Omega_H)\co (H_L\ot \eta_{H_L}\ot H)$
\item[ ]$=r_H$
\item[ ]$=r_H\co \Omega_H$
\item[ ]$=b_H\co \mu_H\co (i_{L}\ot H),$
\end{itemize}
and $H_L\hookrightarrow H$ is a weak $H$-Galois extension with normal basis.

Finally,  in this case, if $H\ot -$ preserves coequalizers, the morphism  $\gamma_H^{-1}$ is almost lineal. Indeed: Let $\varphi_{H\ot_{H_L}H}:H\ot H\ot_{H_L}H\rightarrow H\ot_{H_L}H$ be the factorization though the coequalizer $H\ot n_{H}$ of the morphism $n_{H}\co (\mu_H\ot H)$, i.e., the morphism such that
\begin{equation}
\label{varphiparaH}
\varphi_{H\ot_{H_L}H}\co (H\ot n_{H})=n_{H}\co (\mu_H\ot H).
\end{equation}
Then, by (a4-3) of Definition \ref{Weak-Hopf-quasigroup}, (\ref{nabladeH}) and (\ref{varphiparaH}), 
\begin{itemize}
\item[ ]$\hspace{0.38cm} \varphi_{H\ot_{H_{L}}\ot H}\co (H\ot (\gamma_H^{-1}\co p_{H\ot H}\co (\eta_H\ot H)))$
\item[ ]$=\varphi_{H\ot_{H_L}H}\co (H\ot n_H)\co (H\ot ((\mu_H\ot H)\co (\Pi_{H}^{R}\ot \lambda_H\ot H)\co (H\ot \delta_H)\co \delta_H))$
\item[ ]$=n_H\co (\mu_H\ot H)\co (H\ot \lambda_H\ot H)\co (H\ot \delta_H)$
\item[ ]$=n_H\co (\mu_H\ot H)\co (H\ot \lambda_H\ot H)\co (\mu_H\ot \delta_H)\co (H\ot \Pi_{H}^{R}\ot H)\co (H\ot \delta_H)$
\item[ ]$=\gamma_H^{-1}\co p_{H\ot H},$
\end{itemize}
and $\gamma_H^{-1}$ is almost lineal.

}
\end{example}

To finish this section we show two technical lemmas that will be useful in order to get the main result of this paper which gives a characterization of weak $H$-Galois extensions with normal basis.

\begin{lemma}
 \label{igualdadesgalois}
Let $H$ be a weak Hopf quasigroup and let $A^{co H}\hookrightarrow A$ be a weak $H$-Galois extension. Then the following equalities hold:
\begin{equation}
\label{igualdadesgalois-1}
\rho^{1}_{A\ot_{A^{co H}}A}\co \gamma_{A}^{-1}=((\gamma_{A}^{-1}\co p_{A\ot H})\ot H)\co (A\ot c_{H,H})\co (A\ot \mu_H\ot H)\co (\rho_A\ot ((\lambda_H\ot H)\co \delta_H))\co i_{A\ot H},
\end{equation}
\begin{equation}
\label{igualdadesgalois-2}
((\gamma_A^{-1}\co p_{A\ot H})\ot H)\co (A\ot \delta_H)=\rho^{2}_{A\ot_{A^{co H}}A}\co \gamma_A^{-1} \co p_{A\ot H}.
\end{equation}

\end{lemma}

\begin{proof} The first equality follows easily from (\ref{gammarho-1})  composing with 
$\gamma_{A}^{-1}\ot H$ on the left and with $\gamma_{A}^{-1}$ on the right. On the other hand, if  we compose in  (\ref{gammarho-2}) with  $\gamma_{A}^{-1}\ot H$ on the left  and with $\gamma_{A}^{-1}\co p_{A\ot H}$ on the right we obtain (\ref{igualdadesgalois-2}). 
\end{proof}

\begin{lemma}
 \label{morfismomA}
Let $H$ be a weak Hopf quasigroup and let $A^{co H}\hookrightarrow A$ be a weak $H$-Galois extension with normal basis. Then there is a unique morphism $m_A:A\ot_{A^{co H}}A\rightarrow A$ such that
\begin{equation}
\label{condicionmA}
m_A\co n_{A}=\mu_A\co (A\ot (((i_{A}\ot \varepsilon_H)\co s_A\co b_A))).
\end{equation}
Moreover, the equalities
\begin{equation}
\label{SegundacondicionmA}
m_A\co \gamma_A^{-1}\co p_{A\ot H}\co \rho_A=(i_A\ot \varepsilon_H)\co s_A\co b_A
\end{equation}
and
\begin{equation}
\label{terceracondicionmA}
\rho_A\co m_A=(m_A\ot H)\co \rho^{1}_{A\ot_{A^{co H}}A}
\end{equation}
hold.

\end{lemma}
         
\begin{proof}
The proof for (\ref{condicionmA}) is similar to the given in Lemma 1.9 of \cite{AFG2} but using (\ref{AsubH-2}) instead of the associativity. On the other hand,
\begin{itemize}
\item[ ]$\hspace{0.38cm} m_A\co \gamma_A^{-1}\co p_{A\ot H}\co \rho_A$
\item[ ]$=m_A\co \gamma_A^{-1}\co \gamma_A\co n_{A}\co (\eta_A\ot A)$
\item[ ]$=m_A\co n_{A}\co (\eta_A\ot A)$
\item[ ]$=(i_A\ot \varepsilon_H)\co s_A\co b_A,$
\end{itemize}
and we have (\ref{SegundacondicionmA}). As far as (\ref{terceracondicionmA}), composing with the coequalizer $n_{A}$ and using (\ref{SegundacondicionmA}),  (\ref{muArhoA-2}), the naturalness of $c$, (\ref{condicionmA}) and (\ref{rho-fact-1}), 
\begin{itemize}
\item[ ]$\hspace{0.38cm} \rho_A\co m_A\co n_{A}$
\item[ ]$=((\rho_A\co \mu_A\co (A\ot i_A))\ot \varepsilon_H)\co  (A\ot  (s_A\co b_A))$
\item[ ]$=(((\mu_A\ot H)\co (A\ot c_{H,A})\co (\rho_A\ot  i_A))\ot \varepsilon_H)\co  (A\ot  (s_A\co b_A))$
\item[ ]$=((m_A\co n_{A})\ot H)\co (A\ot c_{H,A})\co (\rho_A\ot A)$
\item[ ]$=(m_A\ot H)\co \rho^{1}_{A\ot_{A^{co H}}A}\co n_{A},$
\end{itemize}
and the equality (\ref{terceracondicionmA}) holds.

\end{proof}

Note that in the previous proof, by the lack of associativity, we cannot say that $m_A$ is a left $A$-module morphism. Nevertheless, if the functor $A\ot -$ preserves coequalizers, by (\ref{AsubH-3}) the equality 
\begin{equation}
\label{msubAdemodulos}
\mu_A\co (A\ot m_A)=m_A\co \varphi_{A\ot_{A^{co H}}A}
\end{equation}
holds.

\section{Cleft extensions associated to a weak Hopf quasigroup}
In this section we introduce the notion of weak H-cleft extension associated to a weak Hopf quasigroup $H$. As a particular instances we recover the theory of cleft extensions  associated to a weak Hopf algebra \cite{nmra1, AFG2} and to a Hopf quasigroup \cite{AFGS-2, AFG-3}.

\begin{definition}
\label{Cleft}
{\rm Let $H$ be a weak Hopf quasigroup and let $(A, \rho_A)$ be a right $H$-comodule magma. We will say that $A^{co H}\hookrightarrow A$ is a weak $H$-cleft extension if there exists a right $H$-comodule morphism $h:H\rightarrow A$
(called the cleaving morphism) and a morphism $h^{-1}:H\rightarrow A$ such that

\begin{itemize}
\item[(c1)] $h^{-1}*h=(A\ot (\varepsilon_H\co \mu_{H}))\co (c_{H,A}\ot H)\co (H\ot (\rho_A\co \eta_A)).$
\item[(c2)] $(A\ot \mu_H)\co (c_{H,A}\ot H)\co (H\ot (\rho_A\co h^{-1}))\co \delta_H=(A\ot \overline{\Pi}_{H}^{R})\co \rho_A\co h^{-1}.$
\item[(c3)] $\mu_{A}\co (\mu_{A}\ot A)\co (A\ot h^{-1}\ot h)\co (A\ot \delta_H)=\mu_{A}\co (A\ot (h^{-1}*h)).$
\item[(c4)] $\mu_{A}\co (\mu_{A}\ot A)\co (A\ot h\ot h^{-1})\co (A\ot \delta_H)=\mu_{A}\co (A\ot (h*h^{-1})).$
\end{itemize}
}
\end{definition}

\begin{example}
\label{Hescleft}
{\rm Let $H$ be a weak Hopf quasigroup. Then $H_L\hookrightarrow H$ is a weak $H$-cleft extension with cleaving map $h=id_H$ and $h^{-1}=\lambda_H$.
}
\end{example}

Note that if $H$ is a weak Hopf algebra  and $(A,\rho_{A})$ is a right $H$-comodule monoid, conditions (c3) and (c4) trivialize. Then, in this case, we get the definition of weak $H$-cleft extension given in \cite{AFG2}. 

On the other hand, as a particular case, if $H$ is a Hopf quasigroup we obtain the following definition of weak $H$-cleft extension:

\begin{definition}
\label{CleftparaHq}
{\rm Let $H$ be a Hopf quasigroup and let $(A, \rho_A)$ be a right $H$-comodule magma. We will say that $A^{co H}\hookrightarrow A$ is a weak $H$-cleft extension if there exists a right $H$-comodule morphism $h:H\rightarrow A$
and a morphism $h^{-1}:H\rightarrow A$ such that

\begin{itemize}
\item[(d1)] $h^{-1}*h=\varepsilon_{H}\ot \eta_{A}.$
\item[(d2)] $(A\ot \mu_H)\co (c_{H,A}\ot H)\co (H\ot (\rho_A\co h^{-1}))\co \delta_H=h^{-1}\ot \eta_H.$
\item[(d3)] $\mu_{A}\co (\mu_{A}\ot A)\co (A\ot h^{-1}\ot h)\co (A\ot \delta_H)=A\ot \varepsilon_H.$
\item[(d4)] $\mu_{A}\co (\mu_{A}\ot A)\co (A\ot h\ot h^{-1})\co (A\ot \delta_H)=\mu_{A}\co (A\ot (h*h^{-1})).$
\end{itemize}
}
\end{definition}
\begin{remark}
\label{cleft-previo}
{\rm Let $H$ be a Hopf quasigroup and let $(A, \rho_A)$ be a right $H$-comodule magma. Let $h:H\rightarrow A$ be a comodule morphism and let $h^{-1}:H\rightarrow A$ be a morphism. Note that, in general,  the convolution product $h*h^{-1}$ is not $\varepsilon_H\ot \eta_A$. If true,  
condition (d4) turns into 
\begin{equation}
\label{d4-new}
\mu_{A}\co (\mu_{A}\ot A)\co (A\ot h\ot h^{-1})\co (A\ot \delta_H)=A\ot \varepsilon_H.
\end{equation}
On the other hand, if we assume  (\ref{d4-new}), we have that $h*h^{-1}=\varepsilon_H\ot \eta_A$ and then 
\begin{equation}
\label{primeraequiv}
\rho_{A}\co h^{-1}=(h^{-1}\ot \lambda_{H})\co c_{H,H}\co \delta_{H}
\end{equation}
holds. Indeed:
\begin{itemize}
\item[ ]$\hspace{0.38cm}(h^{-1}\ot \lambda_{H})\co c_{H,H}\co \delta_{H}$
\item[ ]$= (\rho_{A}\co (h\ast h^{-1}))\ast ((h^{-1}\ot \lambda_{H})\co c_{H,H}\co \delta_{H})$
\item[ ]$= \mu_{A\ot H}\co (( \mu_{A\ot H}\co ((\rho_{A}\co h^{-1})\ot (\rho_{A}\co h))\co \delta_{H})\ot  ((h^{-1}\ot \lambda_{H})\co c_{H,H}\co \delta_{H}) )\co \delta_{H}$
\item[ ]$= (\mu_{A}\ot H)\co (A\ot c_{H,A})\co (\mu_{A}\ot (\mu_{H}\co (\mu_{H}\ot \lambda_{H})\co \delta_{H})\ot A )\co (A\ot c_{H,A}\ot H\ot A)$
\item[ ]$\hspace{0.38cm}\co ((\rho_{A}\co h^{-1})\ot ((h\ot H)\co \delta_{H})\ot h^{-1}) \co (H\ot \delta_{H})\co \delta_{H} $
\item[ ]$=  (\mu_{A}\ot H)\co (\mu_{A}\ot c_{H,A})\co (A\ot c_{H,A}\ot A)\co ((\rho_{A}\co h^{-1})\ot ((h\ot h^{-1})\co \delta_{H}))\co \delta_{H}$
\item[ ]$=\rho_{A}\co h^{-1}.$
\end{itemize}
In the last equalities,  the first one follows by $h*h^{-1}=\varepsilon_H\ot \eta_A$ and the second one by (\ref{chmagma}). In the third one we used that $h$ is a comodule morphism, the coassociativity of $\delta_{H}$ and the naturalness of $c$. The fourth one is a consequence of the quasigroup structure of $H$ and, finally, the last one follows by the naturalness  of $c$ and (\ref{d4-new}).

If (\ref{primeraequiv}) holds, we obtain (d2) because, using the coassociativity of $\delta_{H}$ and the naturalness  of $c$:
\begin{itemize}
\item[ ]$\hspace{0.38cm}  (A\ot \mu_H)\co (c_{H,A}\ot H)\co (H\ot (\rho_A\co h^{-1}))\co \delta_H $
\item[ ]$=(A\ot \mu_{H})\co (c_{H,A}\ot H)\co (H\ot ( (h^{-1}\ot \lambda_{H})\co c_{H,H}\co \delta_{H}))\co \delta_{H} $
\item[ ]$= (h^{-1}\ot H)\co c_{H,H}\co ((id_{H}\ast \lambda_{H})\ot H)\co \delta_{H} $
\item[ ]$= h^{-1}\ot \eta_{H}.$
\end{itemize}

Therefore, if $h*h^{-1}=\varepsilon_H\ot \eta_A$ and $h$ is total ($h\co\eta_{H}=\eta_{A}$), we recover the notion of cleft comodule algebra (or $H$-cleft extension for Hopf quasigroups) introduced in \cite{AFGS-2}.

}
\end{remark}

In the following Proposition we collect the main properties of weak $H$-cleft extensions.

\begin{proposition}
\label{propiedades basicas}
Let $H$ be a weak Hopf quasigroup and let $A^{co H}\hookrightarrow A$ be a weak $H$-cleft extension with cleaving morphism $h$. Then we have that

\begin{itemize}
\item[(i)] The morphisms $h*h^{-1}$ and $q_A=\mu_{A}\co (A\ot h^{-1})\co \rho_A$ factorize through the equalizer $i_{A}.$
\item[(ii)] $\mu_A\co ((h^{-1}*h)\ot A)=(A\ot (\varepsilon_H \co \mu_H))\co (c_{H,A}\ot H)\co (H\ot \rho_A).$
\item[(iii)] $(h^{-1}*h)*h^{-1}=h^{-1}=h^{-1}*(h*h^{-1}).$
\item[(iv)] $h*(h^{-1}*h)=h=(h*h^{-1})*h.$
\item[(v)] $\mu_{A}\co (A\ot (h^{-1}*h))\co \rho_A=id_A.$
\item[(vi)] If $A^{co H}\hookrightarrow A$ satisfies (\ref{AsubH-2}), the equality $\mu_A\co (\mu_A\ot A)\co (A\ot q_A\ot h)\co (A\ot \rho_A)=\mu_A$ holds. 
\end{itemize}

\end{proposition}

\begin{proof}
(i) Taking into account that $h$ is a morphism of right $H$-comodules, $h*h^{-1}=q_A\co h$ and then it suffices to get the proof for the morphism $q_A$.

\begin{itemize}
\item[ ]$\hspace{0.38cm} \rho_A\co q_A$
\item[ ]$=\mu_{A\ot H}\co (\rho_A\ot (\rho_A\co h^{-1}))\co \rho_A$
\item[ ]$=(\mu_A\ot H)\co (A\ot ((A\ot \mu_H)\co (c_{H,A}\ot H)\co (H\ot (\rho_A\co h^{-1}))\co \delta_{H}))\co \rho_A$
\item[ ]$=(\mu_A\ot H)\co (A\ot ((A\ot \overline{\Pi}_{H}^{R})\co \rho_A\co h^{-1}))\co \rho_A$
\item[ ]$=(\mu_A\ot H)\co (A\ot ((A\ot (\overline{\Pi}_{H}^{R}\co \overline{\Pi}_{H}^{R}))\co \rho_A\co h^{-1}))\co \rho_A$
\item[ ]$=(A\ot \overline{\Pi}_{H}^{R})\co \rho_A\co q_A.$
\end{itemize}

In these computations, the first and the second equalities follow because $A$ is a right $H$-comodule magma; the third one by (c2) of Definition \ref{Cleft}; the fourth one relies on the idempotent character of $\overline{\Pi}_{H}^{R}$; finally, the last equality uses the arguments of the preceding identities but in the inverse order.

As a consequence, there is a morphism $p_A:A\rightarrow A^{co H}$ such that $q_A=i_A\co p_A$.

Assertion (ii) is a direct consequence of (c1) of Definition \ref{Cleft}, (b4) of Definition \ref{H-comodulomagma}, (\ref{mu-pi-l}) and the naturalness of $c$. Indeed:
\begin{itemize}
\item[ ]$\hspace{0.38cm} \mu_A\co ((h^{-1}*h)\ot A)$
\item[ ]$= \mu_A\co (((A\ot (\varepsilon_H\co \mu_{H}))\co (c_{H,A}\ot H)\co (H\ot (\rho_A\co \eta_A)))\ot A) $
\item[ ]$= (A\ot (\varepsilon_{H}\co \mu_{H}))\co (c_{H,A}\ot H)\co (H\ot (((\mu_{A}\co c_{A,A})\ot H)\co (A\ot (\rho_{A}\co \eta_{A})))) $
\item[ ]$=(A\ot (\varepsilon_{H}\co \mu_{H}))\co (c_{H,A}\ot H)\co (H \ot ((A\ot \Pi_{H}^{L})\co \rho_{A}))  $
\item[ ]$=  (A\ot (\varepsilon_{H}\co \mu_{H}))\co (c_{H,A}\ot H)\co (H \ot \rho_{A}). $
\end{itemize}

As far as (iii), we get $(h^{-1}*h)*h^{-1}=h^{-1}*(h*h^{-1})$ by (c4) of Definition \ref{Cleft} and by the coassociativity of $\delta_{H}$. The equality $(h^{-1}*h)*h^{-1}=h^{-1}$ follows by (ii) and (c2) of Definition \ref{Cleft}. In a similar way, $h*(h^{-1}*h)=(h*h^{-1})*h$ is a consequence of the coassociativity of $\delta_{H}$ and (c3) of Definition \ref{Cleft}. The equality $h*(h^{-1}*h)=h$ follows using that $h$ is a comodule morphism,  (c1) of Definition \ref{Cleft}  and (\ref{chmagma}). It is easy to prove (v) taking into account (c1) of Definition \ref{Cleft} and (\ref{chmagma}). Finally, by (\ref{AsubH-2}), the condition of right $H$-comodule for $A$, (c3) of Definition \ref{Cleft} and (v), we have 
\begin{itemize}
\item[ ]$\hspace{0.38cm} \mu_A\co (\mu_A\ot A)\co (A\ot q_A\ot h)\co (A\ot \rho_A)$
\item[ ]$=\mu_A\co (\mu_A\ot A)\co (A\ot (i_A\co p_A)\ot h)\co (A\ot \rho_A)$
\item[ ]$=\mu_A\co (A\ot \mu_A)\co (A\ot (i_A\co p_A)\ot h)\co (A\ot \rho_A)$
\item[ ]$=\mu_A\co (A\ot (\mu_{A}\co (\mu_{A}\ot A)\co (A\ot h^{-1}\ot h)\co (A\ot \delta_H)))\co (A\ot \rho_A)$
\item[ ]$=\mu_A\co (A\ot (\mu_A\co (A\ot (h^{-1}*h))\co \rho_A))$
\item[ ]$=\mu_A,$
\end{itemize}
and the proof is complete.

\end{proof}

\begin{remark}
{\rm Note that, in the previous result, we did not use (c4) of Definition \ref{Cleft}.
}
\end{remark} 

\begin{proposition}
\label{equivalencia1}
Let $H$ be a weak Hopf quasigroup and let $(A, \rho_A)$ be a right $H$-comodule magma satisfying (\ref{AsubH-2}). Assume that there exist $h:H\rightarrow A$ and $h^{-1}:H\rightarrow A$ such that  $h$ is a right $H$-comodule morphism and conditions (c1), (c3) and (c4) of Definition
 \ref{Cleft} hold. Then condition (c2) is equivalent to (\ref{primeraequiv}).
\end{proposition}
\begin{proof}
First we will prove (c2)$\Rightarrow$ (\ref{primeraequiv}): Let $f$, $g$ and $l$ be the morphisms $f=(h^{-1}\ot \lambda_H)\co c_{H,H}\co \delta_H$, $g=\rho_A\co h$ and $l=\rho_A\co h^{-1}$. We will show that $f=l$. First of all, note that
\begin{itemize}
\item[ ]$\hspace{0.38cm} f*g$

\item[ ]$=(A\ot \mu_H)\co (c_{H,A}\ot H)\co (\lambda_H\ot (h^{-1}*h)\ot H)\co (H\ot \delta_H)\co \delta_H$

\item[ ]$=(A\ot \mu_H)\co (c_{H,A}\ot (((\varepsilon_H\co \mu_H)\ot H)\co (H\ot c_{H,H})\co (\delta_H\ot H)))\co (\lambda_H\ot c_{H,A}\ot H)$
\item[ ]$\hspace{0.38cm} \co (\delta_H\ot (\rho_A\co \eta_A))$

\item[ ]$=(A\ot \mu_H)\co (c_{H,A}\ot \mu_H)\co (\lambda_H\ot c_{H,A}\ot H)\co (\delta_H\ot ((A\ot \Pi_{H}^{L})\co \rho_A\co \eta_A))$

\item[ ]$=(A\ot (\mu_H\co (\lambda_H\ot \mu_H)\co (\delta_H\ot H)))\co (c_{H,A}\ot H)\co (H\ot (\rho_A\co \eta_A))$

\item[ ]$=(A\ot (\mu_H\co (\Pi_{H}^{R}\ot H)))\co (c_{H,A}\ot H)\co (H\ot (\rho_A\co \eta_A))$

\item[ ]$=(A\ot H\ot (\varepsilon_H\co \mu_H\co c_{H,H}))\co (A\ot (\mu_H\co c_{H,H})\ot c_{H,H})
\co (((\rho_A\co \eta_A)\ot (\delta_H\co \eta_H))\ot H)$

\item[ ]$=\rho_A\co (A\ot (\varepsilon_H\co \mu_{H}))\co (c_{H,A}\ot H)\co (H\ot (\rho_A\co \eta_A))$

\item[ ]$=\rho_A\co (h^{-1}*h)$

\item[ ]$=l*g,$
\end{itemize}
where the first equality follows because $h$ is a comodule morphism as well as by the coassociativity of $\delta_{H}$ and the naturalness of $c$; the second one follows by (c1) of Definition \ref{Cleft}, the coassociativity of $\delta_{H}$ and the naturalness of $c$; in the third one we use (\ref{mu-pi-l}), and the fourth one is a consequence of (b6) of Definition \ref{H-comodulomagma} and the naturalness of $c$. The fifth equality relies on (a4-4) of Definition \ref{Weak-Hopf-quasigroup}, the sixth one on (\ref{mu-pi-r}) and the 
naturalness of $c$ and the seventh one follows because $A$ is a right $H$-comodule and by  the naturalness of $c$. Finally, the eight equality is a consequence of (c1) of Definition \ref{Cleft} and the last one follows by (\ref{chmagma}).

On the other hand, the following identity holds 
\begin{equation}
\label{aux-10}
(h^{-1}*h)\co \mu_H=((\varepsilon_H\co \mu_H)\ot (h^{-1}*h))\co (H\ot \delta_H).
\end{equation}
Indeed: using (c1) of Definition \ref{Cleft}, the naturalness of $c$ and (a2) of Definition \ref{Weak-Hopf-quasigroup}, 
\begin{itemize}
\item[ ]$\hspace{0.38cm} (h^{-1}*h)\co \mu_H$
\item[ ]$= (A\ot (\varepsilon_{H}\co \mu_{H}))\co (c_{H,A}\ot H)\co (\mu_{H}\ot (\rho_{A}\co \eta_{A})) $
\item[ ]$=  (A\ot (\varepsilon_{H}\co \mu_{H}\co (\mu_{H}\ot H)))\co  (c_{H,A}\ot H\ot H)\co (H\ot c_{H,A}\ot H)\co (H\ot H\ot (\rho_{A}\co \eta_{A}))$
\item[ ]$=(A\ot (((\varepsilon_{H}\co \mu_{H})\ot (\varepsilon_{H}\co \mu_{H}))\co (H\ot \delta_{H}\ot H))\co  (c_{H,A}\ot H\ot H)\co (H\ot c_{H,A}\ot H)\co (H\ot H\ot (\rho_{A}\co \eta_{A}))  $
\item[ ]$=  (A\ot (\varepsilon_{H}\co \mu_{H}))\co ((\varepsilon_{H}\co \mu_{H})\ot c_{H,A}\ot H)\co (H\ot \delta_{H}\ot (\rho_{A}\co \eta_{A}))$
\item[ ]$=((\varepsilon_H\co \mu_H)\ot (h^{-1}*h))\co (H\ot \delta_H).  $
\end{itemize}

Then, $(f*g)*f=f$ because
\begin{itemize}
\item[ ]$\hspace{0.38cm} (f*g)*f $
\item[ ]$= (\mu_{A}\ot H)\co (A\ot c_{H,A})\co (A\ot (\mu_H\co(\mu_H\ot \lambda_H)\co (H\ot \delta_H))\ot A)\co ((c_{H,A}\co (\lambda_{H}\ot (h^{-1}\ast h)))\ot H\ot h^{-1})$
\item[ ]$\hspace{0.38cm}\co (H\ot ((\delta_{H}\ot H)\co \delta_{H}))\co \delta_{H} $
\item[ ]$=   (\mu_{A}\ot H)\co (A\ot c_{H,A})\co (A\ot \mu_H\ot A)\co ((c_{H,A}\co (\lambda_{H}\ot (h^{-1}\ast h)))\ot H\ot h^{-1})$ 
\item[ ]$\hspace{0.38cm}\co (H\ot ((((H\ot \Pi_{H}^{L})\co \delta_{H})	\ot H)\co \delta_{H}))\co \delta_{H}$
\item[ ]$=  (\mu_{A}\ot H)\co (A\ot c_{H,A})\co (A\ot \mu_H\ot A)\co ((c_{H,A}\co (\lambda_{H}\ot ((h^{-1}\ast h)\co \mu_{H})))\ot H\ot h^{-1}) $
\item[ ]$\hspace{0.38cm}\co (H\ot ((H\ot c_{H,H})\co ((\delta_{H}\co \eta_{H})\ot H))\ot H)\co (\delta_{H}\ot H)\co \delta_{H}$
\item[ ]$=  (\mu_{A}\ot H)\co (A\ot c_{H,A})\co (A\ot \mu_H\ot A)\co ((c_{H,A}\co (\lambda_{H}\ot (((\varepsilon_H\co \mu_H)\ot (h^{-1}*h))\co (H\ot \delta_H))))\ot H\ot h^{-1}) $
\item[ ]$\hspace{0.38cm}\co (H\ot ((H\ot c_{H,H})\co ((\delta_{H}\co \eta_{H})\ot H))\ot H)\co (\delta_{H}\ot H)\co \delta_{H}$
\item[ ]$=(A\ot \mu_{H})\co (c_{H,A}\ot H)\co (H\ot c_{H,A})\co (\lambda_{H}\ot \Pi_{H}^{L}\ot ((h^{-1}\ast h)\ast h^{-1})) \co (\delta_{H}\ot H)\co \delta_{H}$
\item[ ]$=c_{H,A}\co ((\lambda_{H}\ast   \Pi_{H}^{L})\ot h^{-1})\co \delta_{H}$
\item[ ]$=c_{H,A}\co (\lambda_{H}\ot    h^{-1})\co \delta_{H}  $
\item[ ]$= f,$
\end{itemize}
where the first equality is a consequence of the coassociativity of $\delta_{H}$, the naturalness of $c$ and the condition of comodule morphism for $h$. The second one follows by (a4-6) of Definition \ref{Weak-Hopf-quasigroup}, the third one follows by (\ref{delta-pi-l}) and the fourth one relies on (\ref{aux-10}). In the fifth one we used the coassociativity of $\delta_{H}$ and the naturalness of $c$. The sixth one can be obtained 
using (iii) of Proposition \ref{propiedades basicas} and the naturalness of $c$, the seventh one  follows by (a4-3) of Definition \ref{Weak-Hopf-quasigroup} and the last one follows by the naturalness of $c$.

As a consequence, $f=l$. Indeed:
\begin{itemize}
\item[ ]$\hspace{0.38cm} f$
\item[ ]$=(f*g)*f$
\item[ ]$=(l*g)*f$
\item[ ]$=(\mu_A\ot H)\co (A\ot c_{H,A})\co (\mu_A\ot(\mu_H\co (\mu_H\ot \lambda_H)\co (H\ot \delta_H))\ot A)$
\item[ ]$\hspace{0.38cm} \co (A\ot c_{H,A}\ot H\ot A)\co ((\rho_A\co h^{-1})\ot ((h\ot H)\co \delta_H)\ot h^{-1})\co (\delta_H\ot H)\co \delta_H$
\item[ ]$=(\mu_A\ot H)\co (A\ot c_{H,A})\co ((\mu_{A\ot H}\co  (\rho_A\ot ((A\ot \Pi_{H}^{L})\co \rho_A)))\ot A)\co(( (h^{-1}\ot h)\co \delta_{H}) \ot h^{-1})\co \delta_H$
\item[ ]$=(\mu_A\ot H)\co (A\ot c_{H,A})\co (((\mu_A\ot H)\co (A\ot c_{H,A})\co (\rho_{A}\ot A))\ot A)\co
(( (h^{-1}\ot h)\co \delta_{H}) \ot h^{-1})\co \delta_H$
\item[ ]$=((\mu_{A}\co (\mu_{A}\ot A)\co (A\ot h\ot h^{-1})\co (A\ot \delta_H))\ot H)\co (A\ot c_{H,H})\co ((\rho_A\co h^{-1})\ot H)\co \delta_H$
\item[ ]$=(\mu_A\co (A\ot (h\ast h^{-1}))\ot H)\co (A\ot c_{H,H})\co ((\rho_A\co h^{-1})\ot H)\co \delta_H$
\item[ ]$=(\mu_A\ot H)\co (A\ot c_{H,A})\co ((\rho_A\co h^{-1})\ot (h\ast h^{-1}))\co \delta_H$
\item[ ]$=\mu_{A\ot H}\co  ((\rho_A\co h^{-1})\ot ((A\ot \Pi_{H}^{L})\co \rho_A\co (h\ast h^{-1})))\co \delta_H$
\item[ ]$=\mu_{A\ot H}\co  ((\rho_A\co h^{-1})\ot (\rho_A\co (h\ast h^{-1})))\co \delta_H$
\item[ ]$=\rho_A\co (h^{-1}\ast (h\ast h^{-1}))$
\item[ ]$=l,$
\end{itemize} 
where the first and the second equalities follow by the  identities previously proved, and the third one is a consequence of the coassociativity of $\delta_{H}$, the naturalness of $c$ and the condition of comodule morphism for $h$. In the fourth equality we used that $h$ is a morphism of comodules and (a4-6) of Definition \ref{Weak-Hopf-quasigroup}, while the fifth and the ninth ones follow by (\ref{muArhoA-22}). The sixth one relies on  the coassociativity of $\delta_{H}$ and the naturalness of $c$, the seventh one on (c4) of Definition \ref{Cleft} and the eighth one follows by naturalness of $c$. In the tenth one we applied (i) of Proposition \ref{propiedades basicas} and the eleventh one relies on (\ref{chmagma}) . Finally, the last one follows by  (iii) of  Proposition \ref{propiedades basicas}. 

Conversely,  (\ref{primeraequiv}) $\Rightarrow$ (c2). Indeed: 
\begin{itemize}
\item[ ]$\hspace{0.38cm}(A\ot \mu_H)\co (c_{H,A}\ot H)\co (H\ot (\rho_A\co h^{-1}))\co \delta_H$
\item[ ]$= (A\ot \mu_H)\co (c_{H,A}\ot H)\co (H\ot ((h^{-1}\ot \lambda_H)\co c_{H,H}
\co \delta_H))\co \delta_H $
\item[ ]$=(h^{-1}\ot \Pi_{H}^{L})\co c_{H,H}\co \delta_H  $
\item[ ]$=(h^{-1}\ot (\overline{\Pi}_{H}^{R}\co
\lambda_{H}))\co c_{H,H}\co \delta_H  $
\item[ ]$=(A\ot \overline{\Pi}_{H}^{R})\co \rho_A\co h^{-1},   $
\end{itemize}
where the first and the fourth equalities follow by (\ref{primeraequiv}), the second one by the coassociativity of $\delta_{H}$ and the naturalness of $c$ and the third one by (\ref{pi-antipode-composition-3}).
\end{proof}

\begin{proposition}
\label{equivalencia2}
Let $H$ be a weak Hopf quasigroup and let $(A, \rho_A)$ be a right $H$-comodule magma satisfying (\ref{AsubH-2}). Assume that there exist $h:H\rightarrow A$ and $h^{-1}:H\rightarrow A$ such that  $h$ is a right $H$-comodule morphism and  conditions (c1), (c2) and (c3) of Definition
 \ref{Cleft} hold. Then condition (c4) is equivalent to
\begin{equation}
\label{segundaequiv}
\mu_A\co (\mu_A\ot h^{-1})\co (A\ot \rho_A)=\mu_A\co (A\ot q_A).
\end{equation}
\end{proposition}
\begin{proof}
We get (c4) of Definition
 \ref{Cleft} by composing with $A\ot h$ in (\ref{segundaequiv}) and using that $h$ is a morphism of $H$-comodules.  
 
As far as the "if" part,
 
\begin{itemize}
\item[ ]$\hspace{0.38cm}\mu_A\co (\mu_A\ot h^{-1})\co (A\ot \rho_A)$

\item[ ]$=\mu_A\co ((\mu_{A}\co ((\mu_{A}\co (A\ot q_{A}))\ot h)\co (A\ot \rho_{A}))\ot h^{-1})\co (A\ot \rho_A)$

\item[ ]$=\mu_A\co (\mu_{A}\ot A)\co ((\mu_{A}\co (A\ot q_{A}))\ot ((h\ot h^{-1})\co \delta_{H}))\co (A\ot \rho_A) $

\item[ ]$=\mu_A\co (\mu_A\ot A)\co (A\ot q_A\ot (h\ast h^{-1}))\co (A\ot \rho_A)$

\item[ ]$=\mu_A\co (A\ot \mu_A)\co (A\ot q_A\ot (h\ast h^{-1}))\co (A\ot \rho_A)$

\item[ ]$=\mu_A\co (A\ot (\mu_{A}\co (\mu_{A}\ot A)\co (A\ot h\ot h^{-1})\co (q_{A}\ot \delta_H)\co \rho_{A}))$

\item[ ]$=\mu_A\co (A\ot \mu_A)\co (A\ot (\mu_{A}\co (\mu_{A}\ot A)\co (A\ot ((h^{-1}\ot h)\co\delta_{H})))\ot h^{-1})\co (A\ot A\ot \delta_{H})\co  (A\ot \rho_A)$

\item[ ]$=\mu_A\co (A\ot \mu_A)\co (A\ot (\mu_A\co (A\ot (h^{-1}\ast h))\co \rho_A)\ot h{-1})\co (A\ot \rho_A)$
\item[ ]$=\mu_A\co (A\ot q_A).$
\end{itemize}

In the preceding computations, the first equality follows by (vi) of Proposition \ref{propiedades basicas}; the second one by the comodule condition for $A$, and the third and fifth ones by (c4) of Definition \ref{Cleft}; in the fourth one we use  (\ref{AsubH-2}) and $q_A=i_A\co p_A$. The sixth equality follows because $A$ is a right $H$-comodule and coassociativity of $\delta_{H}$; the seventh  one relies on  (c3) of Definition \ref{Cleft}; finally, in the last one we use (v) of Proposition 
\ref{propiedades basicas}.

\end{proof}

\section{The main theorem}

Now we get the main result of this paper which gives a characterization of Galois extensions with normal basis in terms of cleft extensions.

\begin{theorem}
\label{caracterizacion}
Let $H$ be a weak Hopf quasigroup and let $(A, \rho_A)$ be a right $H$-comodule magma satisfying (\ref{AsubH-2}), (\ref{AsubH-3}) and such that the functor $A\ot -$ preserves coequalizers. The following assertions are equivalent.

\begin{itemize}
\item[(i)] $A^{co H}\hookrightarrow A$ is a weak $H$-Galois extension with normal basis and the morphism $\gamma_{A}^{-1}$ is almost lineal.
\item[(ii)] $A^{co H}\hookrightarrow A$ is a weak $H$-cleft extension.
\end{itemize}

\end{theorem}

\begin{proof}
(i) $\Rightarrow$ (ii) Let $A^{co H}\hookrightarrow A$ be a weak $H$-Galois extension with normal basis. Using that $\Omega_A$ is a morphism of left $A^{co H}$-modules and right $H$-comodules it is  not difficult to see that so are the morphisms $\omega_A=b_A^{-1}\co r_A:A^{co H}\ot H\rightarrow A$ and $\omega_A^{\prime}=s_A\co b_A:A\rightarrow A^{co H}\ot H$. Now define 
$$h=\omega_A\co (\eta_{A^{co H}}\ot H).$$ Taking into account that $\omega_A$ is a morphism of $H$-comodules, so is $h$.

Let $h^{-1}$ be the morphism defined as 
$$h^{-1}=m_A\co \gamma_{A}^{-1}\co p_{A\ot H}\co (\eta_A\ot H),$$
 where $m_A$ is the morphism obtained in Lemma \ref{morfismomA}. By Proposition \ref{monoidecoinvariantes}, (\ref{condicionmA}), and taking into account that $\omega_A^{\prime}$ is a morphism of $H$-comodules we obtain that 
 $$(m_A\ot H)\co \rho^{2}_{A\ot_{A^{co H}}A}\co n_{A}=(\mu_A\ot H)\co (A\ot ((i_A\ot H)\co \omega_A^{\prime}))$$ and then, by (\ref{AsubH-2}) and using that $\omega_{A}$ is a morphism of $A^{co H}$-comodules, we get that 
\begin{itemize}
\item[ ]$\hspace{0.38cm} \mu_A\co (m_A\ot (\omega_A\co (\eta_{A^{co H}}\ot H)))\co \rho^{2}_{A\ot_{A^{co H}}A}\co n_{A}$
\item[ ]$=\mu_A\co ((\mu_A\co (A\ot i_A)\ot (\omega_A\co (\eta_{A^{co H}}\ot H))))\co (A\ot \omega_A^{\prime})$
\item[ ]$=\mu_A\co (A\ot (\omega_A\co \omega_A^{\prime}))$
\item[ ]$=\mu_A.$
\end{itemize}
As a consequence,
\begin{equation}
\label{igualdadmsubAomega}
\overline{\mu}_A=\mu_A\co (m_A\ot h)\co \rho^{2}_{A\ot_{A^{co H}}A},
\end{equation}
where $\overline{\mu}_A$ denotes the factorization of the morphism $\mu_A$ through the coequalizer $n_{A}$, i.e., $\overline{\mu}_A\co n_{A}=\mu_A$. Note that 
\begin{equation}
\label{igualdadmsubAomega-2}
\overline{\mu}_A=(A\ot \varepsilon_{H})\co i_{A\ot H}\co \gamma_{A}
\end{equation}
also holds.

Now we show conditions (c1)-(c4) of Definition \ref{Cleft}. Using (\ref{igualdadesgalois-2}), (\ref{igualdadmsubAomega}) and the equality (\ref{igualdadmsubAomega-2}), we get (c1). Indeed,

\begin{itemize}
\item[ ]$\hspace{0.38cm} h^{-1}*h$
\item[ ]$=\mu_A\co (m_A\ot h)\co \rho^{2}_{A\ot_{A^{co H}}A}\co \gamma_{A}^{-1}\co p_{A\ot H}\co (\eta_A\ot H)$
\item[ ]$=\overline{\mu}_A\co \gamma_{A}^{-1}\co p_{A\ot H}\co (\eta_A\ot H)$
\item[ ]$=(A\ot \varepsilon_H)\co i_{A\ot H}\co \gamma_A\co \gamma_{A}^{-1}\co p_{A\ot H}\co (\eta_A\ot H)$
\item[ ]$=(A\ot \varepsilon_H)\co \nabla_A\co (\eta_A\ot H)$
\item[ ]$=(A\ot (\varepsilon_H\co \mu_{H}))\co (c_{H,A}\ot H)\co (H\ot (\rho_A\co \eta_A)).$
\end{itemize}

The proof for  (c2) is the following: In one hand we have

\begin{itemize}
\item[ ]$\hspace{0.38cm} (A\ot \mu_H)\co (c_{H,A}\ot H)\co (H\ot (\rho_A\co h^{-1}))\co \delta_H$

\item[ ]$=(A\ot \mu_H)\co (c_{H,A}\ot H)\co (H\ot ((m_A\ot H)\co \rho^{1}_{A\ot_{A^{co H}}A}\co \gamma_{A}^{-1}\co p_{A\ot H}\co (\eta_A\ot H)))\co \delta_H$

\item[  ]$=(A\ot \mu_H)\co (c_{H,A}\ot H)
\co (H\ot (((m_A\co \gamma_{A}^{-1}\co p_{A\ot H})\ot H)\co (A\ot c_{H,H})\co (A\ot (\mu_{H}\co (H\ot \lambda_{H}))\ot H)$
\item[ ]$\hspace{0.38cm}\co (\rho_{A}\ot \delta_{H}) \co \nabla_{A}\co (\eta_A\ot H))) \co \delta_{H}$

\item[  ]$= ((m_A\co \gamma_{A}^{-1}\co p_{A\ot H})\ot \mu_{H})\co (A\ot c_{H,H}\ot H)\co (c_{H,A}\ot (c_{H,H}\co 
((\mu_{H}\ot H)\co (H\ot ( (\lambda_{H}\ot H)\co\delta_{H}\co \mu_{H})))))$
\item[ ]$\hspace{0.38cm}\co (H\ot \rho_{A}\ot H\ot H)\co (H\ot c_{H,A}\ot H)\co (\delta_{H}\ot (\rho_{A}\co \eta_{A}))$

\item[  ]$= ((m_A\co \gamma_{A}^{-1}\co p_{A\ot H})\ot \mu_{H})\co (A\ot c_{H,H}\ot H)\co (c_{H,A}\ot (c_{H,H}\co 
((\mu_{H}\ot H)$
\item[ ]$\hspace{0.38cm}\co (H\ot (( (\mu_{H}\co c_{H,H}\co (\lambda_{H}\ot \lambda_{H}))\ot \mu_{H})\co \delta_{H\ot H})))))\co (H\ot \rho_{A}\ot H\ot H)\co (H\ot c_{H,A}\ot H)$
\item[ ]$\hspace{0.38cm}\co (\delta_{H}\ot (\rho_{A}\co \eta_{A}))$

\item[  ]$=((m_A\co \gamma_{A}^{-1}\co p_{A\ot H})\ot H)\co (A\ot c_{H,H})\co (A\ot \mu_{H}\ot H)\co 
(c_{H,A}\ot ((H\ot (\mu_{H}\co (\lambda_{H}\ot H)))\co (\delta_{H}\ot H)) \ot \mu_{H})$
\item[ ]$\hspace{0.38cm}\co (H\ot \rho_{A}\ot \lambda_{H}\ot H\ot H)\co (H\ot c_{H,A}\ot H\ot H)\co (\delta_{H}\ot c_{H,A}\ot H)\co (\delta_{H}\ot (\rho_{A}\co \eta_{A}))$

\item[  ]$= ((m_A\co \gamma_{A}^{-1}\co p_{A\ot H})\ot H)\co (A\ot c_{H,H})\co (A\ot (\mu_{H}\co (H\ot 
(\mu_{H}\co (\Pi_{H}^{L}\ot H))))\ot H)\co (c_{H,A}\ot H\ot H\ot \mu_{H})$
\item[ ]$\hspace{0.38cm}\co (H\ot \rho_{A}\ot \lambda_{H}\ot H\ot H)\co (H\ot c_{H,A}\ot H\ot H)\co (\delta_{H}\ot c_{H,A}\ot H)\co (\delta_{H}\ot (\rho_{A}\co \eta_{A}))$

\item[  ]$= ((m_A\co \gamma_{A}^{-1}\co p_{A\ot H})\ot H)\co (A\ot c_{H,H})\co (A\ot (\mu_{H}\co ((\mu_{H}\co ((H\ot \Pi_{H}^{L})\ot H))))\ot H)\co (c_{H,A}\ot H\ot H\ot \mu_{H})$
\item[ ]$\hspace{0.38cm}\co (H\ot \rho_{A}\ot \lambda_{H}\ot H\ot H)\co (H\ot c_{H,A}\ot H\ot H)\co (\delta_{H}\ot c_{H,A}\ot H)\co (\delta_{H}\ot (\rho_{A}\co \eta_{A}))$

\item[  ]$= ((m_A\co \gamma_{A}^{-1}\co p_{A\ot H})\ot H)\co (A\ot c_{H,H})$
\item[ ]$\hspace{0.38cm}\co (A\ot (\mu_{H}\co (((\varepsilon_{H}\co\mu_{H})\ot H)\co (H\ot c_{H,H})\co (\delta_{H}\ot H))\ot H)\ot H)\co (c_{H,A}\ot H\ot H\ot \mu_{H})$
\item[ ]$\hspace{0.38cm}\co (H\ot \rho_{A}\ot \lambda_{H}\ot H\ot H)\co (H\ot c_{H,A}\ot H\ot H)\co (\delta_{H}\ot c_{H,A}\ot H)\co (\delta_{H}\ot (\rho_{A}\co \eta_{A}))$

\item[  ]$= ((m_A\co \gamma_{A}^{-1}\co p_{A\ot H})\ot H)\co (A\ot c_{H,H})\co 
(((((A\ot (\varepsilon_{H}\co \mu_{H}))\co (c_{H,A}\ot H)\co (H\ot \rho_{A}))\ot \Pi_{H}^{L})$
\item[ ]$\hspace{0.38cm}\co (H\ot c_{H,A})\co (\delta_{H}\ot A))\ot \mu_{H})\co (H\ot c_{H,A}\ot H)\co (\delta_{H}\ot (\rho_{A}\co \eta_{A}))$

\item[  ]$= ((m_A\co \gamma_{A}^{-1}\co p_{A\ot H})\ot H)\co (A\ot c_{H,H})\co (((A\ot (((\varepsilon_{H}\co \mu_{H})\ot \Pi_{H}^{L})\co (H\ot c_{H,H})\co (\delta_{H}\ot H)))$
\item[ ]$\hspace{0.38cm}\co (c_{H,A}\ot H)\co (H\ot \rho_{A}))\ot \mu_{H})\co (H\ot c_{H,A}\ot H)\co (\delta_{H}\ot (\rho_{A}\co \eta_{A}))$

\item[  ]$= ((m_A\co \gamma_{A}^{-1}\co p_{A\ot H})\ot H)\co (A\ot c_{H,H})\co (((A\ot (\Pi_{H}^{L}\co\mu_{H}\co (H\ot \Pi_{H}^{L})))\co (c_{H,A}\ot H)\co (H\ot \rho_{A}))\ot \mu_{H})$
\item[ ]$\hspace{0.38cm}\co (H\ot c_{H,A}\ot H)\co (\delta_{H}\ot (\rho_{A}\co \eta_{A}))$

\item[  ]$= ((m_A\co \gamma_{A}^{-1}\co p_{A\ot H})\ot H)\co (A\ot c_{H,H})\co (((A\ot (\Pi_{H}^{L}\co\mu_{H}))\co (c_{H,A}\ot H)\co (H\ot \rho_{A}))\ot \mu_{H})$
\item[ ]$\hspace{0.38cm}\co (H\ot c_{H,A}\ot H)\co (\delta_{H}\ot (\rho_{A}\co \eta_{A}))$

\item[  ]$= ((m_A\co \gamma_{A}^{-1}\co p_{A\ot H})\ot H)\co (A\ot c_{H,H})\co (A\ot \Pi_{H}^{L}\ot H)\co 
(A\ot (\mu_{H\ot H}\co (\delta_{H}\ot \delta_{H})))$
\item[ ]$\hspace{0.38cm}\co (c_{H,A}\ot H)\co (H\ot (\rho_{A}\co \eta_{A}))$

\item[  ]$= ((m_A\co \gamma_{A}^{-1}\co p_{A\ot H})\ot H)\co (A\ot c_{H,H})\co (A\ot \Pi_{H}^{L}\ot H)\co 
(A\ot (\delta_{H}\co \mu_{H})) \co (c_{H,A}\ot H)\co (H\ot (\rho_{A}\co \eta_{A})),$
\end{itemize}
where the first equality follows by (\ref{terceracondicionmA}), the second one follows by (\ref{igualdadesgalois-1}) and the naturalness of $c$, the third one follows by the naturalness of $c$ and the unit properties and the fourth one is a consequence of (a1) of Definition \ref{Weak-Hopf-quasigroup} and (\ref{anti-antipode-1}). The fifth and the thirteenth equalities rely on the comodule condition for $A$ and on the naturalness of $c$. In the sixth one we used (a4-5) of Definition \ref{Weak-Hopf-quasigroup} and the seventh one follows by (\ref{monoid-hl-1}). The eighth and the eleventh ones are a consequence of (\ref{mu-pi-l}) and the ninth one was obtained using  the naturalness of $c$ and the coassociativity of $\delta_{H}$. The tenth one follows by the naturalness of $c$ and the twelfth one relies on (\ref{pi-delta-mu-pi-1}). Finally, the last one follows by (a1) of Definition \ref{Weak-Hopf-quasigroup}.

On the other hand, 

\begin{itemize}

\item[ ]$\hspace{0.38cm}  (A\ot \overline{\Pi}_{H}^{R})\co \rho_A\co h^{-1}$ 

\item[ ]$=(m_A\ot \overline{\Pi}_{H}^{R})\co \rho^{1}_{A\ot_{A^{co H}}A}\co \gamma_{A}^{-1}\co p_{A\ot H}\co (\eta_A\ot H)$

\item[ ]$=((m_A\co \gamma_{A}^{-1}\co p_{A\ot H})\ot \overline{\Pi}_{H}^{R})\co (A\ot c_{H,H})\co (A\ot \mu_H\ot H)\co (\rho_A\ot ((\lambda_H\ot H)\co \delta_H))\co \nabla_A\co (\eta_A\ot H)$

\item[ ]$=((m_A\co \gamma_{A}^{-1}\co p_{A\ot H})\ot \overline{\Pi}_{H}^{R})\co (A\ot c_{H,H})\co (A\ot \mu_H\ot H)\co (\rho_A\ot ((\lambda_H\ot H)\co \delta_H\co \mu_H))$
\item[ ]$\hspace{0.38cm} \co (c_{H,A}\ot H)\co (H\ot (\rho_A\co \eta_A))$

\item[ ]$=((m_A\co \gamma_{A}^{-1}\co p_{A\ot H})\ot \overline{\Pi}_{H}^{R})\co (A\ot c_{H,H})\co 
(A\ot (\mu_H\co (H\ot \mu_{H})\co (H\ot \lambda_{H}\ot H)\co (\delta_{H}\ot H))\ot H)$
\item[ ]$\hspace{0.38cm} \co (\rho_A\ot \lambda_H\ot \mu_{H})\co (A\ot \delta_{H}\ot H)\co (c_{H,A}\ot H)\co (H\ot (\rho_A\co \eta_A))$

\item[ ]$=((m_A\co \gamma_{A}^{-1}\co p_{A\ot H})\ot \overline{\Pi}_{H}^{R})\co (A\ot c_{H,H})\co 
(A\ot (\mu_H\co (\Pi_{H}^{L}\ot H))\ot H) \co (\rho_A\ot \lambda_H\ot \mu_{H})$
\item[ ]$\hspace{0.38cm}\co (A\ot \delta_{H}\ot H)\co (c_{H,A}\ot H)\co (H\ot (\rho_A\co \eta_A))$

\item[ ]$=((m_A\co \gamma_{A}^{-1}\co p_{A\ot H})\ot H)\co (A\ot c_{H,H})\co 
(A\ot (\overline{\Pi}_{H}^{R}\co \mu_H\co ((\overline{\Pi}_{H}^{R}\co \lambda_{H})\ot H))\ot H)\co (\rho_A\ot \lambda_H\ot \mu_{H})$
\item[ ]$\hspace{0.38cm} \co (A\ot \delta_{H}\ot H)\co (c_{H,A}\ot H)\co (H\ot (\rho_A\co \eta_A))$

\item[ ]$=((m_A\co \gamma_{A}^{-1}\co p_{A\ot H})\ot H)\co (A\ot c_{H,H})\co 
(A\ot (\overline{\Pi}_{H}^{R}\co \lambda_{H}\co \mu_H\co c_{H,H})\ot H)\co (\rho_A\ot H\ot \mu_{H})$
\item[ ]$\hspace{0.38cm} \co (A\ot \delta_{H}\ot H)\co (c_{H,A}\ot H)\co (H\ot (\rho_A\co \eta_A))$

\item[ ]$=((m_A\co \gamma_{A}^{-1}\co p_{A\ot H})\ot H)\co (A\ot c_{H,H})\co 
(A\ot (\Pi_{H}^{L}\co \mu_H\co c_{H,H})\ot H)\co (\rho_A\ot H\ot \mu_{H})$
\item[ ]$\hspace{0.38cm} \co (A\ot \delta_{H}\ot H)\co (c_{H,A}\ot H)\co (H\ot (\rho_A\co \eta_A))$

\item[  ]$= ((m_A\co \gamma_{A}^{-1}\co p_{A\ot H})\ot H)\co (A\ot c_{H,H})\co (A\ot \Pi_{H}^{L}\ot H)\co 
(A\ot (\mu_{H\ot H}\co (\delta_{H}\ot \delta_{H})))$
\item[ ]$\hspace{0.38cm}\co (c_{H,A}\ot H)\co (H\ot (\rho_{A}\co \eta_{A}))$

\item[  ]$= ((m_A\co \gamma_{A}^{-1}\co p_{A\ot H})\ot H)\co (A\ot c_{H,H})\co (A\ot \Pi_{H}^{L}\ot H)\co 
(A\ot (\delta_{H}\co \mu_{H})) \co (c_{H,A}\ot H)\co (H\ot (\rho_{A}\co \eta_{A})).$
\end{itemize}

In the preceding computations, the first equality follows by  (\ref{terceracondicionmA}), the second one by (\ref{igualdadesgalois-1}); in the third we use the unit properties and the fourth one follows by (a1) of Definition \ref{Weak-Hopf-quasigroup}, the comodule condition for $A$, and the naturalness of $c$. The fifth one is a consequence of (a4-5); the sixth one follows by (\ref{pi-antipode-composition-3}) and the naturalness of $c$. In the seventh one we applied (\ref{anti-antipode-1}) and the equality 
\begin{equation}
\label{pil-mu-pirvar-h}
\overline{\Pi}_{H}^{R}\co \mu_{H}\co (\overline{\Pi}_{H}^{R}\ot H)=\overline{\Pi}_{H}^{R}\co \mu_{H},
\end{equation}
which is a consequence of (\ref{pi-delta-mu-pi-2}), (\ref{pi-composition-3}) and (\ref{pi-composition-4}).
The eighth one relies on (\ref{pi-antipode-composition-3}),  and the ninth one follows by the comodule condition for $A$ and the naturalness of $c$. Finally, the last one follows by (a1) of Definition \ref{Weak-Hopf-quasigroup}. 

Therefore, (c2) holds, because 
$$(A\ot \mu_H)\co (c_{H,A}\ot H)\co (H\ot (\rho_A\co h^{-1}))\co \delta_H$$
$$=((m_A\co \gamma_{A}^{-1}\co p_{A\ot H})\ot H)\co (A\ot c_{H,H})\co (A\ot \Pi_{H}^{L}\ot H)\co 
(A\ot (\delta_{H}\co \mu_{H})) \co (c_{H,A}\ot H)\co (H\ot (\rho_{A}\co \eta_{A}))$$
$$=(A\ot \overline{\Pi}_{H}^{R})\co \rho_A\co h^{-1}.$$

To see (c3),
\begin{itemize}
\item[ ]$\hspace{0.38cm} \mu_{A}\co (\mu_{A}\ot A)\co (A\ot h^{-1}\ot h)\co (A\ot \delta_H)$
\item[ ]$=\mu_A\co ((m_A\co \varphi_{A\ot_{A^{co H}}A}\co (A\ot (\gamma_{A}^{-1}\co p_{A\ot H}\co 
( \eta_A\ot H))))\ot h)\co (A\ot \delta_{H})$
\item[ ]$=\mu_A\co ((m_A\co \gamma_{A}^{-1}\co p_{A\ot H})\ot h)\co (A\ot \delta_H)$
\item[ ]$=\mu_A\co (m_A\ot h)\co \rho^{2}_{A\ot_{A^{co H}}A}\co \gamma_{A}^{-1}\co p_{A\ot H}$                                               
\item[ ]$=\overline{\mu}_A\co \gamma_{A}^{-1}\co p_{A\ot H}$
\item[ ]$=(A\ot \varepsilon_H)\co \nabla_A$
\item[ ]$=(\mu_A\ot ((\varepsilon_H\co \mu_{H}))\co (c_{H,A}\ot H)\co (H\ot (\rho_A\co \eta_A)))$
\item[ ]$=\mu_{A}\co (A\ot (h^{-1}*h)),$
\end{itemize}
where the first equality follows by (\ref{msubAdemodulos}); the second one because $\gamma_{A}^{-1}$ is almost lineal (see (\ref{almostlineal})); in the third one we use (\ref{igualdadesgalois-2}); in the fourth one (\ref{igualdadmsubAomega}). The fifth one is a consequence of the equality (\ref{igualdadmsubAomega-2}); the sixth one relies on the definition of $\nabla_A$; and the last equality follows by (c1).

Finally, by (\ref{msubAdemodulos}), the condition of  almost lineal 
 for $\gamma_{A}^{-1}$ and (\ref{can-fact}), we have
\begin{itemize}
\item[ ]$\hspace{0.38cm} \mu_{A}\co (\mu_{A}\ot h^{-1})\co (A\ot \rho_A)$
\item[ ]$=m_A\co \varphi_{A\ot_{A^{co H}}A}\co (A\ot (\gamma_{A}^{-1}\co 
p_{A\ot H}\co (\eta_A\ot H)))\co (\mu_{A}\ot H)\co (A\ot \rho_A)$
\item[ ]$=m_A\co \gamma_{A}^{-1}\co p_{A\ot H}\co (\mu_A\ot H)\co (A\ot \rho_A)$
\item[ ]$=m_A\co n_{A}.$ 
\end{itemize}
Moreover,  by (\ref{msubAdemodulos}), the condition of  almost lineal 
 for $\gamma_{A}^{-1}$, (\ref{condicionmA}), and (\ref{SegundacondicionmA}) we obtain 
\begin{itemize}
\item[ ]$\hspace{0.38cm}\mu_{A}\co (A\ot q_A)$
\item[ ]$= m_{A}\co \varphi_{A\ot_{A^{co H}}A}\co (A\ot (\varphi_{A\ot_{A^{co H}}A}\co (A\ot (\gamma_{A}^{-1}\co p_{A\ot H}\co (\eta_A\ot H)))))\co (A\ot \rho_{A})$
\item[ ]$= m_{A}\co \varphi_{A\ot_{A^{co H}}A}\co (A\ot (\gamma_{A}^{-1}\co p_{A\ot H}\co \rho_A))$
\item[ ]$=\mu_A\co (A\ot (m_A\co \gamma_{A}^{-1}\co p_{A\ot H}\co \rho_A))$
\item[ ]$=m_A\co n_{A}.$
\end{itemize}
Therefore, by Proposition \ref{equivalencia2}, (c4) holds.

Now we will prove (ii) $\Rightarrow$ (i). Let $A^{co H}\hookrightarrow A$ be a weak $H$-cleft extension with cleaving morphism $h$. Then the morphism 
$$\gamma_{A}^{-1}=n_{A}\co (\mu_A\ot A)\co (A\ot ((h^{-1}\ot h)\co \delta_H))\co i_{A\ot H}$$
is the inverse of $\gamma_A$. Indeed, first note that by (c1) of Definition \ref{Cleft} we have 
\begin{equation}
\label{aux-fin-1}
\mu_A\co (A\ot (h^{-1}\ast h))=(A\ot \varepsilon_{H})\co \nabla_{A},
\end{equation}
and, as a consequence, using that $\nabla_{A}$ is a right $H$-comodule morphism, we obtain 
\begin{equation}
\label{aux-fin-2}
((\mu_A\co (A\ot (h^{-1}\ast h))\ot H)\co (A\ot \delta_{H})= \nabla_{A}.
\end{equation}

Then, $ \gamma_{A}\co \gamma_{A}^{-1}=id_{A\square H}$ because 
\begin{itemize}
\item[ ]$\hspace{0.38cm} i_{A\ot H}\co \gamma_{A}\co \gamma_{A}^{-1}$
\item[ ]$=\nabla_{A}\co (\mu_A\ot H)\co (A\ot (\rho_A\co h))\co ((\mu_A\co (A\ot h^{-1}))\ot H)\co (A\ot \delta_H)\co i_{A\ot H}$
\item[ ]$=\nabla_{A}\co ((\mu_A\co (\mu_A\ot A)\co (A\ot ((h^{-1}\ot h)\co \delta_H)))\ot H)\co (A\ot \delta_H)\co i_{A\ot H}$
\item[ ]$=\nabla_{A}\co ((\mu_A\co (A\ot (h^{-1}\ast h))\ot H)\co (A\ot \delta_{H})\co  i_{A\ot H}$                                      
\item[ ]$=\nabla_{A}\co \nabla_A\co i_{A\ot H}$
\item[ ]$=\nabla_A\co i_{A\ot H}$
\item[ ]$=i_{A\ot H},$
\end{itemize}
where the first equality follows by (\ref{can-fact}), the second one  taking into account that $h$ is a morphism of $H$-comodules and the coassociativity of $\delta_{H}$, the third one relies on  (c3) of Definition \ref{Cleft} and the fourth one follows by (\ref{aux-fin-2}). Finally the last equalities follow by the properties of $\nabla_{A}$.

The equality $ \gamma_{A}^{-1}\co \gamma_{A}=id_{A\ot_{A^{co H}}A}$ holds because 

\begin{itemize}
\item[ ]$\hspace{0.38cm} \gamma_{A}^{-1}\co \gamma_{A}\co n_{A}$
\item[ ]$=n_{A}\co (\mu_A\ot A)\co (A\ot (((h^{-1}\ot h)\co \delta_H)))\co \nabla_A\co (\mu_A\ot H)\co (A\ot \rho_A)$
\item[ ]$=n_{A}\co (\mu_A\ot A)\co (A\ot (((h^{-1}\ot h)\co \delta_H)))\co (\mu_A\ot H)\co (A\ot \rho_A)$
\item[ ]$=n_{A}\co ((\mu_A\co (\mu_A\ot h^{-1})\co (A\ot \rho_A))\ot h)\co (A\ot \rho_A)$
\item[ ]$=n_{A}\co ((\mu_A\co (A\ot q_A))\ot h)\co (A\ot \rho_A)$
\item[ ]$=n_{A}\co (A\ot (\mu_A\co (q_A\ot h)\co \rho_A))$
\item[ ]$=n_{A}\co (A\ot (\mu_A\co (\mu_{A}\ot A)\co (A\ot ((h^{-1}\ot h)\co \delta_{H})\co \rho_{A})))	$
\item[ ]$=n_{A}\co (A\ot (\mu_A\co (A\ot (h^{-1}*h))\co \rho_{A}))$
\item[ ]$=n_{A},$
\end{itemize}
where the first equality follows by  (\ref{can-fact}); the second one by (\ref{nabla-3}); in the third and the sixth ones we use that $A$ is a right $H$-comodule; the fourth one relies on Proposition \ref{equivalencia2}. The fifth equality follows because $q_A=i_A\co p_A$; the seventh one uses (c3) of Definition \ref{Cleft}; finally, the last one follows by (v) of Proposition \ref{propiedades basicas}.

Now we show that $\gamma_{A}^{-1}$ is almost lineal. Indeed, firstly note that 
\begin{itemize}
\item[ ]$\hspace{0.38cm} \varphi_{A\ot_{A^{co H}}A}\co (A\ot (\gamma_A^{-1}\co p_{A\ot H}\co (\eta_A\ot H)))$
\item[ ]$=n_{A}\co (\mu_A\ot A)\co (A\ot \mu_A\ot A)\co (A\ot A\ot ((h^{-1}\ot h)\co \delta_{H}))\co 
(A\ot (\nabla_{A}\co (\eta_{A}\ot H)))$
\item[ ]$=n_{A}\co ((\mu_A\co (A\ot (\mu_{A}\co (A\ot h^{-1})\co \nabla_{A}\co (\eta_{A}\ot H))))\ot h)\co 
(A\ot \delta_{H})$
\item[ ]$=n_{A}\co ((\mu_A\co (A\ot ((h^{-1}\ast h)\ast h^{-1})))\ot h)\co 
(A\ot \delta_{H}) $
\item[ ]$=n_{A}\co ((\mu_A\co (A\ot h^{-1}))\ot h)\co 
(A\ot \delta_{H}), $
\end{itemize}
where the first equality follows by the definition of $\gamma_{A}^{-1}$ and (\ref{varphi}), the second one because 
$\nabla_{A}$ is a right $H$-comodule morphism, the third one relies on (\ref{aux-fin-2}) and the last one follows by (iii) of Proposition \ref{propiedades basicas}.

Secondly, by similar arguments, and using (i) of Proposition \ref{propiedades basicas} and (\ref{AsubH-2}), we obtain 
\begin{itemize}
\item[ ]$\hspace{0.38cm} \gamma_A^{-1}\co p_{A\ot H}$
\item[ ]$=n_{A}\co (\mu_A\ot A)\co (A\ot ((h^{-1}\ot h)\co \delta_H))\co \nabla_A$
\item[ ]$=n_{A}\co (\mu_A\ot A)\co ( (\mu_{A}\co (A\ot (h^{-1}\ast h)))\ot ((h^{-1}\ot h)\co \delta_H))\co (A\ot \delta_{H})$ 
\item[ ]$=n_{A}\co ((\mu_A\co (A\ot ((h^{-1}\ast h)\ast h^{-1})))\ot h)\co 
(A\ot \delta_{H}) $
\item[ ]$=n_{A}\co ((\mu_A\co (A\ot h^{-1}))\ot h)\co 
(A\ot \delta_{H}). $
\end{itemize}
Therefore, $\gamma_{A}^{-1}$ is almost lineal.

To finish the proof we must show that the extension has a normal basis. Let $\omega_A$ and $\omega_{A}^{\prime}$ be the morphisms $\omega_A=\mu_A\co (i_A\ot h)$ and $\omega_{A}^{\prime}=(p_A\ot H)\co \rho_A$. By (c3) of Definition \ref{Cleft}, the comodule condition for $\rho_{A}$ and (v) of Proposition (\ref{propiedades basicas}), $\omega_A\co \omega_{A}^{\prime}=id_A$ and then the morphism $\Omega_A=\omega_{A}^{\prime}\co \omega_A:A^{co H}\ot H\rightarrow A^{co H}\ot H$ is idempotent. Let $s_A:A^{co H}\times H\rightarrow A^{co H}\ot H$ and $r_A:A^{co H}\ot H\rightarrow A^{co H}\times H$ be the morphisms such that $s_A\co r_A=\Omega_A$ and $r_A\co s_A=id_{A^{co H}\times H}$. Taking into account that $h$ is a comodule morphism and (\ref{muArhoA-1}), it is not difficult to see that $\Omega_A$ is a morphism of right $H$-comodules.  Also, 
\begin{equation}
\label{omega-ia}
(i_{A}\ot H)\co \Omega_{A}=(\mu_{A}\ot H)\co (i_{A}\ot ((h\ast h^{-1})\ot H)\co \delta_{H}),
\end{equation}
because 
\begin{itemize}
\item[ ]$\hspace{0.38cm} (i_{A}\ot H)\co \Omega_{A}$
\item[ ]$= ((q_{A}\co \mu_{A})\ot H)\co (i_{A}\ot ((h\ot H)\co \delta_{H})$
\item[ ]$= ((\mu_{A}\co (\mu_{A}\ot h^{-1})\co (A\ot (\rho_{A}\co h)))\ot H)\co (i_{A}\ot  \delta_{H}) $ 
\item[ ]$=((\mu_{A}\co (\mu_{A}\ot A)\co (A\ot ((h\ot h^{-1})\co \delta_{H})))\ot H)\co (i_{A}\ot  \delta_{H}) $  
\item[ ]$=  (\mu_{A}\ot H)\co (i_{A}\ot ((h\ast h^{-1})\ot H)\co \delta_{H}),$
\end{itemize}
where the first equality follows by the definition of $\Omega_{A}$, the second one follows by (\ref{muArhoA-1}), in the third one we use the condition of comodule morphism  for $h$,   and the last one relies on (c4) of Definition \ref{Cleft}. 

To prove that $\Omega_{A}$ is a morphism of left $A^{co H}$-modules, first note that, by 
(\ref{muArhoA-1}), (\ref{segundaequiv}) and (\ref{mu-coinv}), we have 
\begin{itemize}
\item[ ]$\hspace{0.38cm} q_{A}\co \mu_{A}\co (i_{A}\ot A)$
\item[ ]$= \mu_{A}\co (A\ot h^{-1})\co \rho_{A}\co \mu_{A}\co (i_{A}\ot A)$
\item[ ]$= \mu_{A}\co (\mu_{A}\ot h^{-1})\co (i_{A}\ot \rho_{A})$
\item[ ]$=\mu_{A}\co (i_{A}\ot q_{A})$
\item[ ]$=\mu_{A}\co (i_{A}\ot (i_{A}\co p_{A})) $
\item[ ]$=i_{A}\co \mu_{A^{co H}}\co ( A^{co H}\ot p_{A}),$
\end{itemize}
and then 
\begin{equation}
\label{old-equa}
p_{A}\co \mu_{A}\co (i_{A}\ot A)=\mu_{A^{co H}}\co ( A^{co H}\ot p_{A}).
\end{equation}
Therefore the equality 
\begin{equation}
\label{omega-fin}
\Omega_{A}=((\mu_{A^{co H}}\co (A^{co H}\ot p_{A}))\ot H)\co ( A^{co H}\ot (\rho_{A}\co h))
\end{equation}
holds because, by (\ref{muArhoA-1}) and (\ref{old-equa}), 
$$\Omega_{A}
=((p_{A}\co\mu_{A} \co (i_{A}\ot A))\ot H)\co ( A^{co H}\ot (\rho_{A}\co h))  
=((\mu_{A^{co H}}\co (A^{co H}\ot p_{A}))\ot H)\co ( A^{co H}\ot (\rho_{A}\co h)).$$

Then, $\Omega_{A}$ is a morphism of left $A^{co H}$-modules. Indeed, by (\ref{omega-fin})
\begin{itemize}
\item[ ]$\hspace{0.38cm}(\mu_{A^{co H}}\ot H)\co (A^{co H}\ot \Omega_{A})$
\item[ ]$=(\mu_{A^{co H}}\ot H)\co  (A^{co H} \ot (((\mu_{A^{co H}}\co (A\ot p_{A}))\ot H)\co ( A^{co H}\ot (\rho_{A}\co h))))$
\item[ ]$= ((\mu_{A^{co H}}\co (A\ot p_{A}))\ot H)\co ( A^{co H}\ot (\rho_{A}\co h))\ot (\mu_{A^{co H}}\ot H)$
\item[ ]$= \Omega_{A}\co  (\mu_{A^{co H}}\ot H).$
\end{itemize}

Finally, let $b_A=r_A\co \omega_{A}^{\prime}$. Using that $\Omega_{A}$ is a right $H$-comodule morphism, we obtain that $b_A$ is a right $H$-comodule morphism. Also, it is easy to show that $b_A$ is an isomorphism with inverse $b_{A}^{-1}=\omega_A\co s_A$. Finally, the morphism $b_A$  is a morphism of left $A^{co H}$-modules because its inverse is a morphism of left $A^{co H}$-modules. Indeed, using that $\Omega_{A}$ is a morphism of left $A^{co H}$-modules, (\ref{mu-coinv}) and (\ref{AsubH-2}), we have 
\begin{itemize}
\item[ ]$\hspace{0.38cm} b_{A}^{-1}\co \varphi_{A^{co H}\times H}$
\item[ ]$= \mu_{A}\co (i_{A}\ot h)\co \Omega_{A}\co (\mu_{A^{co H}}\ot H)\co (A^{co H}\ot s_{A})$
\item[ ]$= \mu_{A}\co ((i_{A}\co \mu_{A^{co H}})\ot h)\co (A^{co H}\ot (\Omega_{A}\co s_{A})) $ 
\item[ ]$=((\mu_{A}\co (\mu_{A}\ot A)\co (i_{A}\ot i_{A}\ot h)\co (A^{co H}\ot s_{A}) $  
\item[ ]$=  \mu_{A}\co (i_{A}\ot (\mu_{A}\co (i_{A}\ot h)))\co (A^{co H}\ot s_{A})$
\item[ ]$=  \mu_{A}\co (i_{A}\ot b_{A}^{-1}).$
\end{itemize}

\end{proof}

\begin{remark}
{\rm  In the associative setting conditions (\ref{AsubH-2}), (\ref{AsubH-3})  hold and, for example, the previous result generalizes the one proved by Doi and Takeuchi for Hopf algebras in \cite{doi3}. Also, for a weak Hopf algebra $H$, by Remark \ref{Galoiswhaandhq}, we obtain that  the assertions
\begin{itemize}
\item[(i)] $A^{co H}\hookrightarrow A$ is a weak $H$-Galois extension with normal basis,
\item[(ii)] $A^{co H}\hookrightarrow A$ is a weak $H$-cleft extension,
\end{itemize}
are equivalent for a right $H$-comodule monoid $A$. This equivalence is a particular instance of the one obtained in \cite{AFG2} for Galois extensions associated to weak entwining structures. 

}
\end{remark}

As a Corollary of Theorem \ref{caracterizacion}, for  Hopf quasigroups we have a result which shows the close connection between the notion of cleft right $H$-comodule algebra ($H$-cleft extension for Hopf quasigroups), introduced in \cite{AFGS-2}, and the one of $H$-Galois extension with normal basis introduced in this paper. Also, when $A^{co H}=K$ we have the equivalence proved in \cite{AFG-3} because, in this case, $i_{A}=\eta_{A}$.

\begin{corollary}
\label{corolarioHq}
Let $H$ be a Hopf quasigroup and let $(A, \rho_A)$ be a right $H$-comodule magma satisfying (\ref{AsubH-2}), (\ref{AsubH-3}) and such that the functor $A\ot -$ preserves coequalizers. The following assertions are equivalent.
\begin{itemize}
\item[(i)] $A^{co H}\hookrightarrow A$ is an $H$-Galois extension with normal basis, the morphism $\gamma_{A}^{-1}$ is almost lineal, $\Omega_{A}=id_{A^{co H}\ot H}$ and $b_{A}\co \eta_{A}=\eta_{A^{co H}}\ot \eta_{H}$.
\item[(ii)] $A^{co H}\hookrightarrow A$ is an $H$-cleft extension.
\end{itemize}
\end{corollary}

\begin{proof} First, note that in this setting $\rho_{A}\co \eta_{A}=\eta_{A}\ot \eta_{H}$ and then $\nabla_{A}=id_{A\ot H}$. Also, the submonoid of coinvariants $A^{co H}$ is defined by the equalizer of $\rho_{A}$ and 
$A\ot \eta_{H}$. Therefore, 
\begin{equation}
\label{ia-hquasi}
\rho_{A}\co i_{A}=i_{A}\ot \eta_{H}.
\end{equation}

The proof for (i) $\Rightarrow$ (ii) is the following. Let $A^{co H}\hookrightarrow A$ be a weak $H$-Galois extension with normal basis. Assume that $\Omega_A=id_{A^{co H}\ot H}$. Then $r_A=id_{A^{co H}\ot H}=s_A$ and by Theorem \ref{caracterizacion}, $A^{co H}\hookrightarrow A$ is a weak $H$-cleft extension with cleaving morphism $h=b_A^{-1}\co (\eta_A\ot H)$, and whose convolution inverse is $h^{-1}=m_A\co \gamma_A^{-1}\co (\eta_A\ot H)$. Moreover,
\begin{itemize}
\item[ ]$\hspace{0.38cm} h*h^{-1}$
\item[ ]$=m_A\co \varphi_{A\ot_{A^{co H}}A}\co (A\ot (\gamma_A^{-1}\co (\eta_A\ot H)))\co \rho_A\co h$
\item[ ]$=m_A\co \gamma_A^{-1}\co \rho_A\co b_A^{-1}\co (\eta_A\ot H)$
\item[ ]$=(i_A\ot \varepsilon_H)\co b_A\co b_A^{-1}\co (\eta_A\ot H)$
\item[ ]$=\eta_A\ot \varepsilon_H,$
\end{itemize}
where the first equality follows because $h$ is a morphism of $H$-comodules and by (\ref{msubAdemodulos}); the second one uses that $\gamma_A^{-1}$ is almost lineal, and the third one relies on (\ref{SegundacondicionmA}). 

Also, $h\co \eta_{H}=\eta_{A}$ because $b_{A}\co \eta_{A}=\eta_{A^{co H}}\ot \eta_{H}$ holds. Therefore, by 
Remark \ref{cleft-previo}, $A^{co H}\hookrightarrow A$ is an $H$-cleft extension.

On the other hand,  let $A^{co H}\hookrightarrow A$ be an   $H$-cleft extension with cleaving morphism $h$. Then, $$h^{-1}\ast h=h\ast h^{-1}=\eta_{A}\ot \varepsilon_{H}$$ because 
$$\mu_{A}\co (\mu_{A}\ot A)\co (A\ot h^{-1}\ot h)\co (A\ot \delta_H)=A\ot \varepsilon_H=\mu_{A}\co (\mu_{A}\ot A)\co (A\ot h\ot h^{-1})\co (A\ot \delta_H).$$

Put $\Omega_A=id_{A^{co H}\ot H}$. Obviously it is an idempotent morphism of left $A^{co H}$-modules and right $H$-comodules. Consider the morphisms $b_A=(p_A\ot H)\co \rho_A$ and  $b_A^{-1}=\mu_A\co (i_A\ot h)$. Using that $A$ is a right $H$-comodule, we obtain
$$ b_{A}^{-1}\co b_A=\mu_{A}\co (\mu_{A}\ot A)\co (A\ot h^{-1}\ot h)\co (A\ot \delta_H)\co \rho_A=id_A.$$

On the other hand, applying that $h$ is a comodule morphism, (\ref{muArhoA-1}) and (\ref{old-equa}), we have 
\begin{itemize}
\item[ ]$\hspace{0.38cm}  b_A\co b_{A}^{-1}$
\item[ ]$=((p_A\co \mu_A\co (i_A\ot A))\ot H)\co (A^{co H}\ot (\rho_{A}\co h))$
\item[ ]$=((p_A\co \mu_A\co (i_A\ot h))\ot H)\co (A^{co H}\ot \delta_H)$
\item[ ]$=((\mu_{A^{co H}}\co (A^{co H}\ot (p_{A}\co h)))\ot H)\co (A^{co H}\ot \delta_H).$
\end{itemize}
Therefore, $b_A\co b_{A}^{-1} = id_{A^{co H}\ot H}$ because 
\begin{equation}
\label{fin-fin}
\mu_{A^{co H}}\co (A^{co H}\ot (p_{A}\co h))=A^{co H}\ot \varepsilon_{H}.
\end{equation}
Indeed, composing with $i_{A}$ we obtain
$$i_{A}\co \mu_{A^{co H}}\co (A^{co H}\ot (p_{A}\co h))=\mu_{A}\co (i_{A}\ot (h\ast h^{-1}))=i_{A}\ot 
\varepsilon_{H}$$
and then (\ref{fin-fin}) is proved.

Trivially, $b_A$ is a morphism of right $H$-comodules, and by (\ref{AsubH-2}), $b_{A}^{-1}$ is a morphism of left $A^{co H}$-modules. Then, $b_A$ is a morphism of left $A^{co H}$-modules.

\end{proof}

\section*{Acknowledgements}
The  authors were supported by  Ministerio de Econom\'{\i}a y Competitividad of Spain (European Feder support included). Grant MTM2013-43687-P: Homolog\'{\i}a, homotop\'{\i}a e invariantes categ\'oricos en grupos y \'algebras no asociativas.

\end{document}